\documentclass[12pt]{amsart}

\usepackage{amsfonts}

\usepackage{amssymb, amsmath, amsthm}
\usepackage{mathtools}

\usepackage[english]{babel}
\pdfoutput=1
\newtheorem{thm}{Theorem}
\newtheorem{lem}{Lemma}[section]
\newtheorem{prop}[lem]{Proposition}
\newtheorem{cor}[lem]{Corollary}

\DeclarePairedDelimiter\ceil{\lceil}{\rceil}
\DeclarePairedDelimiter\floor{\lfloor}{\rfloor}

\DeclareMathOperator{\Tr}{Tr}

\newcommand{\C}{{\bf{C}}}

\newcommand{\Q}{{\bf{Q}}}
\newcommand{\K}{{\bf{K}}}
\newcommand{\Z}{{\bf{Z}}}
\newcommand{\N}{{\bf{N}}}

\newcommand{\SL}{{\bf{SL}}}

\newcommand{\GL}{{\bf{GL}}}
\newcommand{\PSL}{{\bf{PSL}}}

\begin{document}
 
\title[Size of discriminants]{Size of discriminants of periodic geodesics of the modular surface}
\author{Fran\c cois Maucourant}
\address{Universit\'e Rennes, IRMAR UMR 6625, Campus de Beaulieu 35042 Rennes cedex -  France}
\email{francois.maucourant@univ-rennes.fr}
\thanks{Univ Rennes, CNRS, IRMAR - UMR 6625, F-35000 Rennes, France}


\maketitle
\begin{abstract}
Pick a random matrix $\gamma$ in $\Gamma=\SL(2,\Z)$. Denote by $\mathcal{O}_\K$ the Dedekind ring generated by its eigenvalues, and let $\Delta_\K$, $\Delta_\gamma$ and $\Delta=\Tr(\gamma)^2-4$ be the respective discriminant of the rings $\mathcal{O}_\K$, the multiplier ring $M(2,\Z)\cap \Q[\gamma]$ and $\Z[\gamma]$. We show that their ratios admit probability limit distributions. In particular, 42\% of the elements of $\Gamma$ have a fundamental discriminant, and $\Z[\gamma]$ is a ring of integers with probability 32\%.
\end{abstract}

\section{Introduction} 

\subsection{Random geodesics}

 Let $\Gamma=\SL(2,\Z)$. To a periodic oriented geodesic on the modular surface $\Gamma\backslash \mathbb{H}^2$ are associated several arithmetic quantities. The purpose of this article is to give an idea of their typical size. To pick a periodic geodesic at random, we will take a representative $\gamma \in \Gamma$ of a conjugacy class $[\pm \gamma]$ in $\PSL(2,Z)$, as follows. For a large parameter $T>0$, consider the ball of radius $T$ in $\Gamma$:
 $$\Gamma_T=\{ \gamma \in {\bf SL}(2,\Z) \, : \, \|\gamma\|\leq T\},$$
where $\|.\|$ is the Frobenius norm $\|\gamma\|=\sqrt{\Tr( ^t \gamma \gamma)}$. A classical result is the asymptotic $|\Gamma_T|\sim 6T^2$. Denote by $\mathbb{P}_T$ the uniform probability measure on $\Gamma_T$. Pick a random matrix $\gamma$ using $\mathbb{P}_T$. \\

 We recall that $\gamma$ is likely to be {\em hyperbolic}, more precisely for all $\eta>0$,
\begin{equation} \label{proba_hyperbolique}
 \mathbb{P}_T( |\Tr(\gamma)|> 2)= 1-O(T^{-1+\eta}).
\end{equation}
It is also known that the trace $\Tr(\gamma)$ has roughly size $T$:
\begin{prop} \label{distrib_trace}
For all $\alpha<\beta$,
\begin{equation}\label{distrib_trace2}
\lim_{T\to +\infty} \mathbb{P}_T\left(\frac{\Tr(\gamma)}{T}\in (\alpha,\beta) \right)= \frac{2}{\pi} \int_{\sup(\alpha,-1)}^{\inf(\beta,1)} \sqrt{1-x^2}\, \mathrm{d}x.
\end{equation}
\end{prop}

Assuming that $\gamma$ is hyperbolic, its eigenvalues $\epsilon_\gamma,\epsilon_\gamma^{-1}$, where $|\epsilon_\gamma|>1$, are units of the ring of integers $\mathcal{O}_\K$ of the real quadratic field $\K=\Q(\sqrt{\Tr(\gamma)^2-4})$. The unit $\epsilon_\gamma$ is close to $\Tr(\gamma)$, and $\epsilon_\gamma/T$ has the same asymptotic distribution (\ref{distrib_trace2}) as $\Tr(\gamma)/T$. In terms of the length $\ell_\gamma=2\log(|\epsilon_\gamma|)$ of the periodic geodesic associated to the conjugacy class $[\pm \gamma]\subset \PSL(2,\Z)$, it is just a bit smaller than $2\log(T)$: for all $\beta\geq 0$, 
\begin{equation}\label{distrib_length}
\lim_{T\to +\infty} \mathbb{P}_T\left( 2\log(T)-\beta \leq \ell_\gamma \leq 2\log(T) \right)= \frac{2}{\pi} \int_{0}^{\beta} \sqrt{e^{-z}-e^{-2z}}\, \mathrm{d}z.
\end{equation}

Our first result is that the eigenvalue $|\epsilon_\gamma|$ is likely to be the fundamental unit $\epsilon_\K$ of $\mathcal{O}_\K$, so $\ell_\gamma$ is twice the regulator $R_\K=\log \epsilon_\K$ of the quadratic field $\K$. In particular, the regulator is essentially $\log(T)+O(1)$.
\begin{thm} \label{proba_fondamental}
For all $\eta>0$,
\begin{equation} 
 \mathbb{P}_T( |\Tr(\gamma)|> 2, \, |\epsilon_\gamma|=\epsilon_\K)= 1-O(T^{-\frac12+\eta}).  
\end{equation}
 In particular, the ring $\mathcal{O}_\K$ is unlikely to have an unit of negative norm.
\end{thm}
 This property is a bit stronger than the trite fact that $\gamma$ is likely to be primitive (i.e not a power in $\Gamma$): it means that $\gamma$ is not a power in $\GL(2,\Q)$, or equivalently, the length of the associated geodesic is not a strict multiple of the length of another periodic geodesic. For example, primitive matrices like 
 $$\left( \begin{array}{cc} 2 & 1 \\ 1 & 1 \end{array} \right)=\left( \begin{array}{cc} 1 & 1 \\ 1 & 0 \end{array} \right)^2 , \;  \left( \begin{array}{cc} 26 & 135 \\ 5 & 26 \end{array} \right)=\left( \begin{array}{cc} 2 & 9 \\ 1/3 & 2 \end{array} \right)^3,$$
 are quite rare. In the second example, the length of the associated geodesic is 3 times the length of the geodesic given by $\pm\left( \begin{array}{cc} 2 & 1 \\ 3 & 2 \end{array} \right)$.
 
\subsection{Discriminants} 
 
 We now introduce other arithmetical data attached to $\gamma$.
The matrix ring $\Q[\gamma]$ is isomorphic to $\K$ and it will be convenient to identify the ring of integers $\mathcal{O}_\K$ with the corresponding subring of $\Q[\gamma]$. The {\em multiplier ring} can be defined by
$$\mathcal{O}_\gamma=\Q[\gamma]\cap M(2,\Z).$$

 The {\em discriminant} $\Delta_\gamma$ of the periodic geodesic corresponding to the conjugacy class of $\pm \gamma$ in $\PSL(2,\Z)$ can be defined as the discriminant of the multiplier ring $\mathcal{O}_\gamma$. The link with more classical definitions using the quadratic form associated to the closed geodesic will be made in section \ref{quadratic_forms}. Let $u_\gamma,f_\gamma$ be the successive indexes in the inclusions 
$$\Z[\gamma] \subset \mathcal{O}_\gamma \subset \mathcal{O}_\K,$$
namely $u_\gamma=\left[\mathcal{O}_\gamma: \Z[\gamma]\right]$,
$f_\gamma=\left[\mathcal{O}_\K: \mathcal{O}_\gamma \right]$. The number $f_\gamma$ is called the {\em conductor} of the multiplier ring. It turns out that if we write $\gamma=\left( \begin{array}{cc} a & b \\ c & d \end{array} \right)$, then
$u_\gamma =\gcd(c,d-a,b)$. Moreover, 
$$\begin {cases}
\Tr(\gamma)^2-u_\gamma^2\Delta_\gamma= 4,\\
f_\gamma^2\Delta_\K=\Delta_\gamma,
\end{cases}$$
where $\Delta_\K$ is the (fundamental) discriminant of $\K$.
Notice that the first equation can be interpreted as a Pell Equation 
$$X^2-\Delta_\K Y^2=4,$$
 and $(|\Tr(\gamma)|,u_\gamma f_\gamma)$ is then, by Theorem \ref{proba_fondamental}, the fundamental solution of this equation  with high probability. Our main result is that $u_\gamma,f_\gamma$ are usually pretty small:
\begin{thm} \label{maintheorem}
The joint distribution of $(u_\gamma,f_\gamma)$ converges to a probability as $T\rightarrow +\infty$: there exist $(q_{n,m})_{n\geq 1,m\geq 1}$ positive numbers such that $\sum_{n,m\geq 1} q_{n,m}=1$,
with the property that for all $(n,m) \in \N^2$,
$$\lim_{T\to +\infty} \mathbb{P}_T\left(u_\gamma=n,\, f_\gamma=m \right)=q_{n,m}.$$

\end{thm}

 In fact, the probability $q_{n,m}$ will be shown to be a product of local factors
\begin{equation}\label{qnm}
  q_{n,m}=\prod_{p \; \mathrm{prime}} \mu_p(E_{p,v_p(n),v_p(m)}),
\end{equation}
where $v_p(n),v_p(m)$ are the $p$-adic valuations of $n,m$, $E_{p,\alpha,\beta}$ some clopen subset in the group of $p$-adic matrices $\SL(2,\Z_p)$, and $\mu_p$ the normalized Haar measure on the latter group. The formulas for these local factors are given in Propositions \ref{formula_for_muE} and \ref{i_hate_two}.

 So, since $u_\gamma$ and $f_\gamma$ are usually small, both the discriminant $\Delta_\gamma$ and the fundamental discriminant $\Delta_\K$ are usually roughly of size $T^2$. Some special values of $n,m$ have interesting interpretation:

\begin{cor} \label{cor1}

74\% of the  quadratic forms
$$\tilde{Q}_\gamma:=\mathrm{sgn}(\Tr(\gamma)) \det\left( \left(\begin{array}{c}X \\ Y\end{array}\right),\gamma \left(\begin{array}{c}X \\ Y \end{array}\right) \right),$$
are primitive, that is 
\begin{align*}
\lim_{T\to \infty} \mathbb{P}_T(\Delta_\gamma=\Tr(\gamma)^2-4)& = \sum_m q_{1,m} \\
& =\frac{5}{6}\prod_{p \; prime,p>2} \left(1-\frac{2}{p(p^2-1)}\right)\simeq 0.7439...
\end{align*}
 \end{cor}

\begin{cor} \label{cor2}
The discriminant $\Delta_\gamma$ is equal to the fundamental discriminant with probability 42\%: 
 \begin{align*}
\lim_{T\to \infty} \mathbb{P}_T(\Delta_\gamma=\Delta_\K)& =\sum_n q_{n,1}\\
&= \frac{75}{112} \prod_{p>2} \left(1- \frac{2p}{p^3-1} \right)\simeq 0,4269...
\end{align*}
 \end{cor}
 In \cite{MR3635355}, Bourgain and Kontorovich proved the existence of fundamental geodesics in thin semigroups, a significantly harder problem than the lattice case we are looking at here. A related problem was recently considered by Bhargava, Shankar, and Wang, who also established a positive density result for square-free discriminants of quadratic forms \cite[Theorem 1]{bhargava2022squarefree}.
 
 \begin{cor} 32\% of matrices $\gamma$ in $\Gamma$ are such that $\Z[\gamma]$ is a Dedekind ring:
\begin{align*}
 \lim_{T\to \infty} \mathbb{P}_T(\Z[\gamma]\simeq \mathcal{O}_\K)&= \frac{7}{12}\prod_{p>2} \left(1-\frac{2(p^2+p-1)}{p^2(p^2-1)}\right)\\ & \simeq 0.3267...
\end{align*}
 \end{cor}
 This is analogous to the following result. For number fields of fixed degree $n$ chosen by picking a random monic integer polynomial $f$ of degree $n$, Bhargava, Shankar, and Wang \cite[Theorem 1.2]{8dbac6ab4ed74641b8e9810eeb5fae5c} established that 
 $\Z[X]/(f)$ is the integer ring of the field with positive probability.

 \subsection{Class numbers}
 
  The last arithmetic quantitity we want to mention are the {\em class numbers}. Recall that the narrow class number $h(\Delta)$ for a discriminant $\Delta$ of a (real) quadratic ring $\mathcal{O}$ is the number of classes of invertible fractional ideals modulo products by element of positive norm. It happens to be also the number of periodic geodesics of $\Gamma\backslash \mathbb{H}^2$ of discriminant $\Delta$. Geodesics with the same discriminant share the same multiplier ring (namely the unique real quadratic ring of the given discriminant), and length, although two geodesics of the same length may have different discriminants, for example the ones given by the conjugacy classes of  $\pm\left( \begin{array}{cc} 49 & 26 \\ 32 & 17 \end{array} \right)$ and $\pm\left( \begin{array}{cc} 33 & 136 \\ 8 & 33 \end{array} \right)$, both primitive of length $2\ln(33+8\sqrt{17})$ and of respective discriminants $1088$ and $68$. The average of class numbers according to length of the geodesics is studied in \cite{MR675187}. A corollary of Theorem \ref{maintheorem} and the Brauer-Siegel Theorem is that those class numbers are typically of size $T^{1\pm \epsilon}$:
 
 \begin{cor}\label{brauer-siegel}
  The quantity $\log h(\Delta_\K)/ \log T$ converges in probability to $1$, that is for all $\eta>0$,
  $$\lim_{T\to +\infty} \mathbb{P}_T\left(\left| \frac{\log h(\Delta_\K)}{\log T}-1 \right| >\eta \right)= 0.$$
  Similarly, the class number $h(\Delta_\gamma)$  satisfies
   for all $\eta>0$,
  $$\lim_{T\to +\infty} \mathbb{P}_T\left(\left| \frac{\log h(\Delta_\gamma)}{\log T}-1 \right| >\eta \right)= 0.$$  
 \end{cor} 

\subsection{Concluding remarks}
 It must be noted that the probability $\mathbb{P}_T$ used here is in a sense "not far" from the uniform probability of the set of  periodic geodesics of length $\leq 2\log(T)$.  
If one restricts to subsets of $\Gamma$ that are invariant by conjugacy, hyperbolic and invariant under  $\gamma\mapsto -\gamma$ (and hence correspond to sets of oriented periodic geodesics), the two family of probabilities share the same sets of positive upper density. I don't know a reference for this folklore fact, so a statement and its proof are included in appendix \ref{appendixB}.\\

 This way to pick a random geodesic gives very large discriminants and class number when compared to the size of the fundamental unit. This phenomenon was called the "{\em discriminant-regulator paradox}" by McMullen in \cite[Chapter 28]{McMullen}.
 The explanation is that number theorists usually order quadratic fields by discriminant, while from the geometrical point of view, the fields appear ordered by the length of geodesics - which is twice regulator. To quote Sarnak \cite{zbMATH03910484}, 
 \begin{quote} 
 "That the class numbers are large on average when ordered by the size of the unit is not surprising, as we are favouring discriminants with small units..."
 \end{quote}

 Another interesting arithmetical quantity associated to a periodic geodesic is the Rademacher invariant. Its behaviour was analyzed by Mozzochi in \cite{Mozzochi2013LinkingNO}.

 We refer the reader to \cite{MR3434881} for practical numerical implementation of picking a matrix at random using $\mathbb{P}_T$. The algorithm described there by Rivin follows a probability slightly different from $\mathbb{P}_T$ (but close enough), and is much more efficient than the naive enumeration of all matrices of norm $\leq T$.

\subsection{Plan of the paper}

 In section \ref{arithmetic}, we discuss the arithmetic objects attached to a hyperbolic  matrix $\gamma$. In section \ref{proofth1}, we prove Theorem \ref{proba_fondamental}: the largest eigenvalue of $\gamma$ is the fundamental unit with probability close to $1$. 

In section \ref{section4}, we discuss the $p$-adic formalism and uniform equidistribution results modulo congruence subgroups that will help to establish Theorem \ref{maintheorem}.
 
 In section \ref{section5}, we give the strategy to prove Theorem \ref{maintheorem}, reducing it to three facts, the main two being that knowledge of the $p$-adic valuation of the indexes $u_\gamma,f_\gamma$ can be determined by looking at the reduction of the matrix $\gamma$ modulo $p^k$, for a suitable $k$, and that the discriminant typically does not have large prime square factors.

 In section \ref{section6}, we establish the first claim about the $p$-adic valuation of the conductors. In section \ref{section7}, we determine the local factors $\mu_p(E_{p,\alpha,\beta})$, by counting  matrices in $\SL(2,\Z/p^k \Z)$ in an elementary fashion. Note that in these two sections, in order to keep things simple, the prime number $p$ is usually assumed to be $\neq 2$, while the somewhat tedious case $p=2$ is postponed to the appendix \ref{appendixA}.

 In section 8, we establish the crucial fact that most discriminants are not divisible by the square of a large prime (Proposition \ref{no_large_prime}), using the uniform counting on congruence subgroups.

 In section 9, we explain how to get Corollary \ref{brauer-siegel}.

 In Appendix  \ref{appendixA}, we detail the computations of the local factors for the exceptional prime $p=2$. In Appendix \ref{appendixB}, we prove that enumerating matrices by norm, or periodic geodesic by length, yield the same negligible sets, for properties which are conjugacy invariants of $\PSL(2,\Z)$.

 \section{Arithmetic quantities associated to a matrix $\gamma$}
\label{arithmetic}

 Let us now review in more detail the arithmetic data we are interested in. 
\subsection{Discriminant}
 
 Let $\gamma \in \Gamma$ be a matrix
$$\gamma=\left( \begin{array}{cc} a & b\\ c& d  \end{array}\right).$$
Its characteristic polynomial is:
$$\chi_\gamma(X)=\det(\gamma-XI_2)=X^2-\Tr(\gamma) X+1.$$ 

 We will assume in what follows that $\gamma$ is hyperbolic, that is $|\Tr(\gamma)|>2$. In this case the polynomial $\chi_\gamma$ is irreducible in $\Q[X]$, 
the field $\K=\Q[\gamma]\simeq \Q[X]/(\chi_\gamma)\simeq \Q(\sqrt{\Tr(\gamma)^2-4})$ containing the eigenvalues of $\gamma$ is a real quadratic field. The positive integer $\Tr(\gamma)^2-4$ might not be square-free, but we can always find the unique integers $m>0$ and square-free $D>1$ such that
 $$\Tr(\gamma)^2-4=m^2D.$$
Then the ring of integers of $\K$ is 
$$\mathcal{O}_\K=\begin{cases}
 \Z[\sqrt{D}] & \; \mathrm{if} \; D\equiv 2,3 \, \mathrm{mod} \;4,\\
 \Z[\frac{1+\sqrt{D}}2] & \; \mathrm{if} \; D\equiv 1 \, \mathrm{mod}\; 4. 
\end{cases}
$$
Its discriminant is fundamental and equal to
$$\Delta_\K=\begin{cases}
 4D & \; \mathrm{if} \; D\equiv 2,3 \, \mathrm{mod} \; 4,\\
 D   &\; \mathrm{if} \; D\equiv 1 \, \mathrm{mod} \; 4. 
\end{cases}
$$
  
 We can also consider the ring $\Z[\gamma]\simeq \Z[X]/(\chi_\gamma)$. The discriminant of this ring is equal to $\Tr(\gamma)^2-4$. We define the multiplier ring $\mathcal{O}_\gamma$ as the stabilizer of $\Z^2$:
 $$\mathcal{O}_\gamma=\{\gamma'\in \Q[\gamma] \, : \, \gamma'(\Z^2)\subset \Z^2 \}=\Q[\gamma]\cap M(2,\Z).$$
 Recall that we defined the discriminant $\Delta_\gamma$ as the discriminant of this ring. All orders, i.e. subrings of $\K$ which are $\Z$-modules of rank $2$, are of the form $\Z[f\sqrt{D}]$ or $\Z[f\frac{1+\sqrt{D}}2]$ depending if $D=2,3$, or $1$ mod $4$, for some integer $f$ called the conductor of the ring. The conductor is also the index of the ring in the maximal order $\mathcal{O}_\K$, and its discriminant is $f^2\Delta_\K$. Here we defined $f_\gamma$ as the conductor of the multiplier ring $\mathcal{O}_\gamma$, so 
 $$\Delta_\gamma=f_\gamma^2\Delta_K.$$
 By definition, $u_\gamma$ is the index of $\Z[\gamma]$ in $\mathcal{O}_\gamma$, so $m=u_\gamma f_\gamma$ and 
 $$\Tr(\gamma)^2-4=u_\gamma^2 \Delta_\gamma.$$
 
 From a practical point of view, given the coefficients $a,b,c,d$ of $\gamma$, $m$ and $D$ are deduced from the trace by factoring possible squares factors of $(a+d)^2-4$, $\Delta_\K$ is computed using $D$, but it remains to identify $u_\gamma$ and $f_\gamma$ as factors of $m$. There is indeed a simple formula for $u_\gamma$:
 
 \begin{lem}\label{Deltareduit}
  The index of $\Z[\gamma]$ in $\mathcal{O}_\gamma$ is given by:
  $$u_\gamma=gcd(c,d-a,b)$$
 \end{lem}
 \begin{proof}   
  The set $\mathcal{O}_\gamma$ is defined by
  $$\mathcal{O}_\gamma=M(2,\Z)\cap \{\lambda I_2+\mu \gamma \, : \, (\lambda,\mu) \in \Q^2 \}.$$
  If we set $u'=gcd(c,d-a,b)$, we claim that it is equal to
  $$\mathcal{O}':=\Z I_2+ \Z \frac{\gamma-aI_2}{u'}.$$
  This ring is obviously has $\Z[\gamma]$ as a subring of index $u'$, so the equality $\mathcal{O}_\gamma=\mathcal{O}'$ is sufficient to conclude that $u_\gamma=u'$.
  By definition of $u'$, $\frac{\gamma-aI_2}{u'}$ has indeed integer coefficients, so $\mathcal{O}_\gamma\subset \mathcal{O}'$. Reciprocally,  
  let $\lambda,\mu \in \Q^2$, such that $\gamma'=\lambda I_2+\mu \gamma$ has integer coefficients,
  and write $\mu=\frac{v}{w}$, with $v,w$ integers, $w>0$ and coprime to $v$.
  Then 
  $$\gamma'= \left( \begin{array}{cc} \lambda + \frac{va}w  & \frac{vb}w \\ \frac{vc}w & \lambda + \frac{vd}w  \end{array}\right),$$
  but since $\lambda + \frac{va}w$ is an integer, say $n$, the matrix
  $$\gamma'-nI_2=\left( \begin{array}{cc} 0  & \frac{vb}w \\ \frac{vc}w & \frac{v(d-a)}w  \end{array}\right)=\frac{v(\gamma-aI_2)}{w},$$
  so $w$ divides $vc,vb,v(d-a)$, and since $v$ is coprime to $w$, $w$ divides $u'$, so $\gamma'-nI_2 \in \Z\frac{(\gamma-aI_2)}{u'}$, and $\gamma'\in \mathcal{O}'$ as
   requested. 
 \end{proof} 
 
\subsection{Quadratic forms}\label{quadratic_forms}

  Although we will not use this, here we make the link with the more classical definition of the discriminant $\Delta_\gamma$ using quadratic forms.
  The quadratic form
$$\tilde{Q}_\gamma(X,Y):=\mathrm{sgn}(\Tr(\gamma)) \det\left( \left(\begin{array}{c}X \\ Y\end{array}\right),\gamma\left(\begin{array}{c}X \\ Y \end{array}\right) \right),$$
which is obviously $\gamma$-invariant with integer coefficients, can be expressed
$$\tilde{Q}_\gamma(X,Y)= \mathrm{sgn}(\Tr(\gamma)) (cX^2+(d-a)XY-bY^2).$$
 Its discriminant (as a quadratic form) is, again, $\Tr(\gamma)^2-4$. 
The factor $\mathrm{sgn}(t_M)$ in the expression is here to insure that $\tilde{Q}_\gamma=\tilde{Q}_{-\gamma}$, so $\tilde{Q}_\gamma$ depends only on the class of $\gamma$ in $\PSL(2,\Z)$. \\

The form $\tilde{Q}_\gamma$ is not necessarily primitive, meaning the gcd of its coefficients - here $u_\gamma$ by Lemma \ref{Deltareduit} -, may not be $1$. So generally one consider instead the quadratic form 
 $$Q_\gamma=\frac{\tilde{Q}_\gamma}{u_\gamma},$$
see for example \cite{MR675187}, \cite{MR2504745}, \cite{MR3635355}, \cite{MR3058601}. The discriminant of this quadratic form is then $\Delta_\gamma$, which is the usual definition for the discriminant of an hyperbolic conjugacy class of $\PSL(2,\Z)$. The discriminant characterize the multiplier ring, that is two primitive hyperbolic matrices have the same discriminant if and only if their multiplier rings are isomorphic. This situation is specific to the quadratic case. The set of geodesics sharing the same discriminant $\Delta$ can be identified with the set of $\SL(2,\Z)$-classes of primitives integer quadratic forms of discriminant $\Delta$. This set also identifies with the narrow class group of their common multiplier ring, that is the group of invertible fractional ideals modulo multiplication by elements of positive norm. The cardinal of this set is called the {\em class number} $h(\Delta)$.

\subsection{Units}
 We denote by 
 $$\{ \epsilon_\gamma,\bar{\epsilon}_\gamma \}=\left\{ \frac{\Tr(\gamma)\pm \sqrt{ \Tr(\gamma)^2-4}}2 \right\},$$
  the two eigenvalues of $\gamma$, chosen so that $|\epsilon_\gamma|>1>|\bar{\epsilon}_\gamma|$. The number $|\epsilon_\gamma|$ is thus a positive unit in the field $\K$.
The quadratic field $\K$ has a {\em fundamental unit} $\epsilon_\K \in \mathcal{O}_\K$, a {\em regulator} $R_\K=\log \epsilon_\K$, and a {\em class number} $h(\Delta_\K)$. The norm of the fundamental unit $N_{\K/\Q}(\epsilon_\K)=\epsilon_\K \bar{\epsilon}_\K$ is either $\pm 1$.
By Dirichlet units' Theorem, there exists an integer $e_\gamma$ such that $|\epsilon_M|=(\epsilon_\K)^{e_\gamma}$.

\subsection{Distribution of the trace} Here we indicate how to recover the formula in  Proposition \ref{distrib_trace} and the asymptotic (\ref{distrib_length}). It may be worth mentioning that these two formulas depends on the choice of the norm $\|.\|$ considered in the definition of $\Gamma_T$, here the Frobenius norm $\| \gamma \|=\sqrt{\Tr(^t\gamma \gamma)}$, while Theorems \ref{proba_fondamental} and \ref{maintheorem} do not.\\

By \cite[Theorem 2]{MR2286635}, applied to the characteristic function 
$\chi_{\alpha,\beta}$ of the set of matrices of Frobenius norm $\leq 1$ and of trace between $\alpha<\beta$,
\begin{equation*}
\begin{split}
\lim_{T\to +\infty} \mathbb{P}_T& \left(\frac{\Tr(\gamma)}T  \in (\alpha,\beta)\right)=\\
& \frac1{2\pi^2} \int_{0}^1\int_{0}^{2\pi}\int_{0}^{2\pi} \chi_{\alpha,\beta}
 \left( \begin{array}{cc} r\sin \theta \sin \varphi & r\cos \theta \sin \varphi \\
r\sin \theta \cos \varphi & r\cos \theta \cos \varphi
\end{array}
\right)r\, {d}r \,\mathrm{d}\theta \, \mathrm{d}\varphi.
\end{split}
\end{equation*}
 Note that \cite[Theorem 2]{MR2286635} assumes the function to be continuous, and $\chi_{\alpha,\beta}$ is not. Here $\chi_{\alpha,\beta}$ can be approximated from above or below by continuous functions having approximately the same integral with respect to the limit measure, so this formula is also valid in this case.
Since the trace of the above matrix is $r\sin(\theta+\varphi)$, the above integral is more easily computed using the change of variables
$$(r,\theta,\varphi)\mapsto (x=r\sin(\theta+\varphi),y=r\cos(\theta+\varphi),\theta),$$
which sends $(0,1)\times (0,2\pi)\times (0,2\pi)$ to $\mathbb{D}\times [0,2\pi]$, the set of matrices of trace in $(\alpha,\beta)$ into $\{x\in (\alpha,\beta), (x,y) \in \mathbb{D}, \theta \in (0,2\pi)\} $ and changes the polar coordinate measure $r\, {d}r \,\mathrm{d}\theta \, \mathrm{d}\varphi$ into $\mathrm{d}x \, \mathrm{d}y \, \mathrm{d}\theta$. We thus get  
$$
\lim_{T\to +\infty} \mathbb{P}_T \left(\frac{\Tr(\gamma)}T  \in (\alpha,\beta)\right)=\frac{2}{\pi} \int_{\sup(\alpha,-1)}^{\inf(\beta,1)} \sqrt{1-x^2}\, \mathrm{d}x,$$
as claimed.\\
 
 To recover the asymptotic (\ref{distrib_length}), we first notice that provided that $\gamma$ is hyperbolic, $| \epsilon_\gamma-\Tr(\gamma)|\leq 1$, so Proposition \ref{distrib_trace} holds for $\frac{\epsilon_\gamma}T$ instead of $\frac{\Tr(\gamma)}T$. Recall that the length is given by $\ell_\gamma=2\log(|\epsilon_\gamma|)$. Then (\ref{distrib_length}) is obtained from the previous formula by the change of variable $z=-2\ln(|x|)$.

\section{Eigenvalues are fundamental units}
\label{proofth1}

 In this section, we prove Theorem \ref{proba_fondamental}, which states that with high probability, the absolute value of the eigenvalue $\epsilon_\gamma$ of the randomly picked matrix $\gamma$ is the fundamental unit of the integer ring of the underlying field $\Q[\gamma]$. The idea is that if that does not happen, then the trace of $\gamma$ must lie in the integer image of a family of polynomials $P_k$ of degree $k\geq 2$, which is a fairly scarce set in $\mathbb{N}$. We first prove three useful lemmata.
 
\subsection{Almost all matrices are hyperbolic}

 The following Lemma is a minor variant of \cite[Lemma 7.1]{MR3635355}.
 \begin{lem}\label{bornetrace}
   For all $\eta>0$, there exists $C_\eta>0$ such that for all $t\in \Z$ and all $T\geq 2$,
  $$|\{\gamma \in \Gamma_T \, : \, \Tr(\gamma)=t\}| \leq C_\eta T^{1+\eta}.$$
 \end{lem}
 \begin{proof}  
 We can and will assume that $|t|\leq 2T$, otherwise the set $\{\gamma \in \Gamma_T \, : \, \Tr(\gamma)=t\}$ is empty so the bound is trivial.\\
 
  Let $\gamma=\left( \begin{array}{cc} a & b\\ c& t-a  \end{array}\right)$ be a integer matrix of norm $\leq T$, trace $t$ and determinant $1$.
 There is at most possible $2T+1$ choices for $a$ in $[-T,T]$. Since the determinant is one, we have $bc=-a^2+at-1$, so for each $a$, the integer $b$ must be chosen among the divisors of $|-a^2+at-1|\leq 3T^2+1\leq 4T^2$. Recall (see e.g. \cite[Thm 315]{MR568909}) that for all $\eta>0$, there exists a $c_\eta>0$ such that the number of divisors $d(N)$ of an integer $N\geq 1$ satisfies $d(N)\leq c_\eta N^{\eta/2}$. So we have $2T+1$ choices for $a$, given $a$ we have at most $c_\eta (4T^2)^{\eta/2}$ choices for $b$, and $c$ is deduced from $a,b,t$. Overall, there is at most $O(T^{1+\eta})$ possible choices, where the implied constant depends only on $\eta$.
 \end{proof}

 When applied to the finite set of values $t\in \{-2,-1,0,1,2\}$, and recalling the well-known asymptotic for $|\Gamma_T|$:
 $$|\Gamma_T|\sim 6T^2,$$
(see for example \cite[Theorem 1.1]{JTNB_2005__17_1_301_0}), we recover the claim (\ref{proba_hyperbolique}) of the introduction: 

 \begin{cor} \label{proba_hyp}
 For any $\eta>0$,
 $$ \mathbb{P}_T( |\Tr(\gamma)|> 2)= 1-O_\eta(T^{-1+\eta}).$$
  \end{cor}

 \subsection{Traces of powers}
 
 \begin{lem} \label{puissancedeux}
 Let $\gamma \in {\bf GL}(2,\C)$ be a matrix such that $\det(\gamma)\in \{\pm 1\}$. Then
  $$\Tr(\gamma^2)=\Tr(\gamma)^2\mp 2.$$
 
 \end{lem} 
  This follows immediately for the characteristic equation $\gamma^2-\Tr(\gamma)\gamma+\det(\gamma)I_2=0$. The trace of a power is in general a polynomial of the trace:
 
\begin{lem} \label{puissancek}
 There exists a sequence $(P_k)_{k\geq 2}$ of polynomials, where $P_k$ is of degree $k$ such that for all matrices $\gamma \in {\bf SL}(2,\C)$,
  $$\Tr(\gamma^k)=P_k(\Tr(\gamma)),$$
and we have for $x> 2$,
 $$P_k(x)\geq (x-1)^{k}.$$
  Moreover, if $k$ is odd and $|x|>2$, then $x$ and $P_k(x)$ share the same sign. In fact, $P_k(X)=2T_k(X/2)$, where $(T_k)_{k\geq 0}$ is the Chebyshev polynomial of the first kind. 
 \end{lem}
 \begin{proof} Define the sequence of polynomials $(Q_k)_{k\geq 0}$ of degree $k$ by the following recursion:
 $$\begin{cases}
 Q_0(X)=1,\\
 Q_1(X)=X,\\
 Q_{k+2}(X)=XQ_{k+1}-Q_{k}, \, k\geq 0.
 \end{cases}$$
 In other words, $Q_k(2X)$ are the Chebyshev polynomials of the second kind.
 If $\gamma \in {\bf SL}(2,\C)$, we claim that 
 \begin{equation} \label{puissance}
 \forall k\geq 2, \;
 \gamma^k=Q_{k-1}(\Tr(\gamma))\gamma - Q_{k-2}(\Tr(\gamma))I_2.
 \end{equation}
 Indeed, writing $t=\Tr(\gamma)$, this is true for $k=2$ by the characteristic equation
 $$\gamma^2=t\gamma- I_2,$$
 and if for some $k\geq 2$,
 $$\gamma^k=Q_{k-1}(t)\gamma - Q_{k-2}(t)I_2,$$
 then, multiplying by $\gamma$ and using the characteristic equation,
 \begin{align*}
 \gamma^{k+1}&=Q_{k-1}(t)(t\gamma-I_2) - Q_{k-2}(t)\gamma,\\
 &=(tQ_{k-1}(t)-Q_{k-2}(t))\gamma-Q_{k-1}(t)I_2,\\
 &=Q_{k}(t)\gamma-Q_{k-1}(t)I_2.
 \end{align*}
Thus Equation (\ref{puissance}) follows by induction. Now we define for $k\geq 2$ the polynomial $P_k$
$$P_k=XQ_{k-1}-2Q_{k-2}=Q_k-Q_{k-2},$$  
and it can be checked that $P_k(2X)/2$ agree with the Chebyshev polynomials of the first kind. By Equation (\ref{puissance}),
$$\Tr(\gamma^k)=P_k(\Tr(\gamma)).$$
 
 Let's now prove the lower bound  $P_k(x)\geq (x-1)^{k}$ for all $x>2$. Let $x>2$, there exists an hyperbolic matrix $\gamma$ in $\SL(2,\C)$ whose trace is $x=\Tr(\gamma)$. Let $\lambda,\lambda^{-1}$ be the two eigenvalues of $\gamma$, inverse of each other since $\det(\gamma)=1$. Up to switching $\lambda$ and its inverse, we can assume that $\lambda>1>\lambda^{-1}>0$, so in particular
 $$\lambda=\Tr(\gamma)-\lambda^{-1} > \Tr(\gamma)-1=x-1,$$
 so for any $k\geq 2$, 
 $$P_k(x)=\Tr(\gamma^k)=\lambda^k+\lambda^{-k}> (x-1)^k.$$
With the same notations, notice that $\lambda,\lambda^{-1}$ have the same sign than $\Tr(\gamma)$ and $\Tr(\gamma^k)$ provided $k$ is odd, so the statement about the sign of $x$ and $P_k(x)$ follows. 

 \end{proof}

\subsection{Proof of Theorem \ref{proba_fondamental}}
 
 Let $T>2$ be large. Consider the sets of matrices of norm less than $T$ such that their trace is in the image the polynomials $X^2+ 2$ and $P_k$ for some integer $\leq 2\sqrt{T}$: 
  $$F_{2,-}^{T}=\{\gamma \in \Gamma_T \, : \exists t\in [1,2\sqrt{T}]\cap \Z, \Tr(\gamma)=t^2+ 2 \},$$  
  $$\mathrm{For} \; k\geq 2, \;F_k^{T}=\{\gamma \in \Gamma_T \, : \exists t\in [3,2\sqrt{T}]\cap \Z, \Tr(\gamma)=P_k(t) \}.$$

Notice that:

\begin{prop} Let $\gamma \in \Gamma$, hyperbolic, such that $|\epsilon_\gamma| \neq \epsilon_\K$ and $\|\gamma\|\leq T$. Then $\gamma$ lies in the set
  $$F_T=\pm \left( F_{2,-}^{T}\cup \bigcup_{2 \leq k \leq 1+\frac{\log T}{\log 2}} F_{k}^{T} \right).$$
  In particular, we have
  $$\mathbb{P}_T( |\Tr(\gamma)|> 2, \, |\epsilon_\gamma|\neq \epsilon_\K)\leq \frac{|F_T|}{|\Gamma_T|} $$
  \end{prop}
  \begin{proof}
  Up to considering $-\gamma$ instead of $\gamma$, we can assume that $\Tr(\gamma)>2$.\\
  
  First consider the case where the integer ring of $\K=\Q[\gamma]$ has a unit with negative norm. So $\epsilon_\gamma$ is necessarily a even power of the fundamental unit, so is a square : there exists $A\in \Q[\gamma]$ in the integer ring of $\K$ with determinant $\pm 1$, such that $A^2=\gamma$. Its trace is an integer since $A$ lies in the ring of integers. We can replace $A$ by $-A$ to insure that $t:= \Tr(A)$ is nonnegative. By Lemma \ref{puissancedeux}, $\Tr(\gamma)=t^2 + 2$. Since $\Tr(\gamma)\leq 2T$, this implies that $t\leq \sqrt{2T-2}\leq 2\sqrt{T}$, so
  $\gamma \in F_{2,-}^{T}$.\\
  
 Now assume that the fundamental unit $\epsilon_\K$ has norm $1$. Since by assumption $\epsilon_\gamma \neq \epsilon_\K$, then $\epsilon_\gamma$ is a $k$-th power of another unit of the integer ring of $\Q[\gamma]$, for some $k\geq 2$. So there exists $A\in \Q[\gamma]$ with integer trace $t=\Tr(A)$ and determinant $1$ such that $A^k=\gamma$, so that
 
$$\Tr(\gamma)=P_k(t).$$

 If $k$ is even, one can replace $A$ by $-A$ and assume that $t\geq 0$. If $k$ is odd, $t$ and $\Tr(\gamma)$ have the same sign so $t\geq 0$ also. In any case, since the matrix $\gamma$ is hyperbolic, so is $A$ and thus $t>2$, so $t\geq 3$. For a upper bound on $t$, using Lemma \ref{puissancek}:
 $$2T\geq \Tr(\gamma)=P_k(t)\geq (t-1)^k,$$
 so
 $$3\leq t \leq \sqrt[k]{2T}+1,$$
which we can bound using the crude estimate $\sqrt[k]{2T}+1\leq 2\sqrt{T}$ for $T>2$ and $k\geq 2$.
Now $T\geq \frac12 (t-1)^k\geq 2^{k-1}$ so 
 $$2\leq k\leq 1+\frac{\log(T)}{\log(2)}.$$

 so $\gamma \in F_{k}^{T}$ with the above bound on $k$. 
 \end{proof}

 We conclude the prove of Theorem \ref{proba_fondamental} by bounding the number of elements in $F_T$. Elements of $F_k^{T}$ have at most $2\sqrt{T}$ different traces. By Lemma \ref{bornetrace}, we have for every $\eta>0$, 
 
 $$|F_{k}^{T}|\leq C_\eta (2\sqrt{T}) T^{1+\eta}= 2 C_\eta T^{3/2+\eta}.$$
Similarly,
 $$|F_{2,-}^{T}|\leq C_\eta (2\sqrt{T}) T^{1+\eta}=
2C_\eta T^{3/2+\eta}.$$
Adding these $2+\log_2(T)$ upper bounds, we get
 $$|F_T|\leq 2(2+\log_2(T)) C_\eta T^{3/2+\eta}$$
 so, decreasing slightly the value of the parameter $\eta$ in order to dismiss the logarithmic term, we get for $T$ large enough,
 $$\mathbb{P}_T( |\Tr(\gamma)|> 2, \, |\epsilon_\gamma|\neq \epsilon_\K)\leq \frac{|F_T|}{|\Gamma_T|} = O_\eta( T^{-\frac12+\eta}).$$
Together with Corollary \ref{proba_hyp}, this concludes the proof of Theorem \ref{proba_fondamental}.

\section{Equidistribution Theorem, and $p$-adic formalism}
\label{section4}
 \subsection{The Uniform Equidistribution Theorem}
 
 Let $\Gamma(N)$ be the congruence subgroup of level $N$, that is the kernel of the surjective reduction map $\pi_N:\SL(2,\Z)\to \SL(2,\Z/N\Z)$. It is classical that elements of $\Gamma_T$ distribute evenly in each class mod $\Gamma(N)$: for any subset $E \subset \SL(2,\Z/N\Z)$,
 $$\lim_{T\to +\infty} \mathbb{P}_T(\pi_N(\gamma) \in E)=\frac{|E|}{[\Gamma:\Gamma(N)]}.$$

  We will need the following refinement, due to Nevo and Sarnak. They showed, using the uniform spectral gap for congruence subgroups, that
\begin{thm}\cite[Theorem 3.2]{Nevo_Sarnak}
\label{equid_GN}
 There exist constants $C>0$, $\beta>0$, $T_0>0$ such that for all integer $N\geq 2$, and all $\gamma_0 \in \Gamma$ and $T\geq T_0$,
$$\left| |\gamma_0 \Gamma(N)\cap \Gamma_T| - \frac{|\Gamma_T|}{[\Gamma:\Gamma(N)]} \right|\leq C |\Gamma_T|^{1-\beta}.$$
\end{thm}

 \subsection{The $p$-adic formalism}
 
 It will be convenient to rephrase the equiditribution Theorem (without the remainder terms given by Nevo and Sarnak) using the vocabulary of $p$-adic and profinite integers. This will allow to avoid reference to the level $N$ that will be implicitly considered. When $p$ is prime, we denote by $\Z_p$ the ring of $p$-adic integers, and $\hat{\Z}\simeq \prod_p \Z_p$ the ring of profinite integers. Let $\mu_p$ be the Haar measure on $\SL(2,\Z_p)$, normalized as a probability, and $\mu=\otimes_p \mu_p$ the Haar measure on $\SL(2,\hat{\Z})\simeq \prod_p \SL(2,\Z_p)$, by virtue of the Chinese Remainder Theorem. Recall that a clopen set (closed and open set) of this profinite group is nothing else that than the preimage of a set by the reduction mod $N$, for some $N$ depending on the clopen set. The set $\Gamma_T$ can be seen as a subset of $\SL(2,\hat{\Z})$, and the equidistribution Theorem can be rewritten:
 \begin{thm} \label{equid_rephrased} (Equidistribution Theorem mod $N$, rephrased) For any clopen set $E\subset \SL(2,\hat{\Z})$, we have
 $$\lim_{T\to+\infty} \mathbb{P}_T(\gamma \in E)=\mu(E).$$
 \end{thm}  

\section{Proof of Theorem \ref{maintheorem}}
\label{section5}

 In this section, we prove Theorem \ref{maintheorem}, assuming
three facts that will be proved later.
\subsection{The three facts} \label{3fact} These are:
 \begin{enumerate}
 \item When $p$ is a prime number, the $p$-adic valuations $v_p(u_\gamma),v_p(f_\gamma)$ can be determined by looking at the reduction of $\gamma$ in $\SL(2,\Z/p^k\Z)$ for sufficiently large $k\geq 1$. More precisely, there is a partition of the $p$-adic matrices of trace $\neq \pm2$ into non-empty clopen sets:
 $$\SL(2,\Z_p)-\Tr^{-1}(\pm 2)= \bigsqcup_{\alpha,\beta\geq 0} E_{p,\alpha,\beta},$$
 such that for any hyperbolic matrix $\gamma$,
 \begin{equation}\label{charac_Epab}
 (v_p(u_\gamma)=\alpha,v_p(f_\gamma)=\beta) \Leftrightarrow \gamma \in E_{p,\alpha,\beta}.
 \end{equation}
 Moreover, the set of matrices of trace $\pm 2$ in $\SL(2,\Z_p)$ has zero $\mu_p$-measure.\\  
 
 \item The following product over all primes is convergent:
 $$\prod_p \mu_p(E_{p,0,0})=\frac{7}{12}\prod_{p\geq 3}\left(1 -\frac{2(p^2+p-1)}{p^2(p^2-1)}\right)>0.$$

 \item For any given $\epsilon>0$, with probability $1-\epsilon$, the product $u_\gamma f_\gamma$ does not have any large prime factors $>K(\epsilon)$, for some function $K(\epsilon)$. This means that for $T$ large enough,
$$\mathbb{P}_T\left(u_\gamma=\prod_{p\leq K(\epsilon)} p^{v_p(u_\gamma)}, 
f_\gamma=\prod_{p\leq K(\epsilon)} p^{v_p(f_\gamma)}\right)\geq 1-\epsilon.$$ 
 
 \end{enumerate}
 These three facts and the Chinese Remainder Theorem imply that to recover $u_\gamma$ and $f_\gamma$, it is (with  probability close to $1$) sufficient to look at the reduction mod $N$, for a suitable  $N$.  The first and second facts are easy and treated in sections 6 and 7, but involve some lengthy elementary counting to get the formulas for the local factors $\mu_p(E_{p,\alpha,\beta})$, especially in the case $p=2$ which is relegated to the appendix. The third one is the crucial step, and uses the spectral gap for congruence subgroups as expressed in Nevo and Sarnak's Theorem \cite{Nevo_Sarnak}, and insures the tightness of the sequence of law of $(u_\gamma,f_\gamma)$ as $T\to+\infty$. This will be done in section 8.
 
\subsection{Proof of Theorem \ref{maintheorem}}
 Assuming the three above facts, we now prove the Theorem. Given $n,m$ positive integers, our goal is to show that for some $q_{n,m}>0$ and for all $\epsilon>0$, when $T$ is large enough,
 $$\left|\mathbb{P}_T(u_\gamma=n,f_\gamma=m)-q_{n,m}\right|\leq 3\epsilon.$$ 
 
Let $K$ be larger than any prime factor of $nm$, such that $K\geq K(\epsilon)$ and moreover
\begin{equation}\label{approx_product}   
  \left| \prod_{p\leq K}\mu_p(E_{p,v_p(n),v_p(m)})-\prod_{p}\mu_p(E_{p,v_p(n),v_p(m)})\right|\leq \epsilon.
  \end{equation}
  This is possible because the probabilities $\mu_p(E_{p,\alpha,\beta})$ are between $0$ and $1$, so the finite product decreases to its limit as $K\to \infty$. Moreover, the non-divergence of the product $\prod_p \mu_p(E_{p,0,0})$ tells us that the limit  
  $$q_{n,m}:=\prod_p \mu_p(E_{p,v_p(n),v_p(m)}),$$
  is positive, as long as all the individual factors are nonzero. This latter property will obvious from their explicit computation in Propositions \ref{formula_for_muE} and \ref{i_hate_two}.\\

 Let 
 $$G_{n,m,K}=\left(\prod_{p\leq K} E_{p,v_p(n),v_p(m)} \right)\times \prod_{p>K} \SL(2,\Z_p) \subset \SL(2, \hat{\Z}).$$
 This is a clopen set, so by the equidistribution Theorem \ref{equid_rephrased},
$$\lim_{ T\to+\infty}  \mathbb{P}_T(\gamma \in G_{n,m,K})=\mu(G_{n,m,K})=\prod_{p\leq K}\mu_p(E_{p,v_p(n),v_p(m)}).$$
In other words, by the characterization (\ref{charac_Epab}) of $E_{p,\alpha,\beta}$,
$$\lim_{ T\to+\infty} \mathbb{P}_T \left(\forall p\leq K, v_p(u_\gamma)=v_p(n), v_p(f_\gamma)=v_p(m) \right)
=\mu(G_{n,m,K}),$$
so for all $T$ large enough,
$$\left|\mathbb{P}_T \left(\forall p\leq K, v_p(u_\gamma)=v_p(n), v_p(f_\gamma)=v_p(m) \right)-
\mu(G_{n,m,K})\right|\leq \epsilon.$$
But since $u_\gamma f_\gamma$ has no large prime factor $p\geq K$ with probability $1-\epsilon$, for $T$ large enough
$$\mathbb{P}_T\left(u_\gamma=\prod_{p\leq K} p^{v_p(u_\gamma)}, 
f_\gamma=\prod_{p\leq K} p^{v_p(f_\gamma)}\right)\geq 1-\epsilon.$$
Since we chose $K$ to be larger than any prime factor of $nm$, we have $n=\prod_{p\leq K} p^{v_p(n)}$ and similarly for $m$. The last two inequalities imply:
$$ \left| \mathbb{P}_T \left(u_\gamma=n, f_\gamma=m \right)-\mu(G_{n,m,K})\right| \leq 2\epsilon.$$

 Recall that $K$ was chosen in order to satisfy Inequality (\ref{approx_product}), which can be rewritten:
  $$|\mu(G_{n,m,K})-q_{n,m}|\leq \epsilon,$$
  so we get for all $T$ large enough,
  $$\left| \mathbb{P}_T \left(u_\gamma=n, f_\gamma=m \right) -q_{n,m}\right|\leq 3\epsilon,$$
 which concludes the proof of the convergence part of Theorem \ref{maintheorem}. Now, we wish to check that $(q_{n,m})_{n,m\geq 1}$ is a probability. Since for fixed $p$, the sets $(E_{p,\alpha,\beta})_{\alpha,\beta\geq 0}$ partition $\SL(2,\Z_p)$ modulo a set of zero $\mu_p$-measure, we have
 $$\sum_{\alpha,\beta \geq 0} \mu_p(E_{p,\alpha,\beta})=1.$$
 A sequence $(\alpha_p)_p$ of integers indexed by prime numbers is almost-zero (a.z.) if $\alpha_p=0$ for sufficiently large $p$. With this notation, because of the convergence of $\prod_{p} \mu_p(E_{p,0,0})$, it follows that for any sequences $(\alpha_p)_p,(\beta_p)_p\geq 0$, the product  $\prod_{p} \mu_p(E_{p,\alpha_p,\beta_p})$ is nonzero if and only if both sequences $(\alpha_p)_p,(\beta_p)_p$ are almost-zero. This implies the formal identities: 
$$1=\prod_p \left( \sum_{\alpha,\beta \geq 0} \mu_p(E_{p,\alpha,\beta}) \right)=\sum_{(\alpha_p)_p,(\beta_p)_p \, a.z.}
\prod_p \mu_p(E_{p,\alpha_p,\beta_p})=\sum_{n,m\geq 1} q_{n,m}.$$

\section{The indexes $u_\gamma$, $f_\gamma$ through reduction mod $p^k$}
\label{section6}

\subsection{The $p$-adic valuations of $u_\gamma$, $p\neq 2$}
 
 Here we express the $p$-adic valuations of $u_\gamma$ and $f_\gamma$ in terms of congruence conditions on the matrix $\gamma$, and define the clopen sets $E_{p,\alpha,\beta}$, for an odd prime $p$. A useful convention is that any congruence condition mod $p^0$ is automatically satisfied. The case of the prime $p=2$ is essentially similar, but more complicated so is deferred to appendix \ref{appendixA}.
 
 \begin{lem}\label{ugamma}
  Let $\gamma \in \Gamma$ be hyperbolic, and let $p$ be prime. The $p$-adic valuation $v_p(u_\gamma)$ of $u_\gamma$ is the largest integer $k$ such that $\gamma \; \mathrm{mod} \; p^k$ is a scalar matrix. In particular,
 for $p\neq 2$, 
 $$v_p(u_\gamma)=\sup \left\{k\geq 0 \; : \; \gamma = \pm I_2 \; \mathrm{mod} \; p^k \right\}.$$

 \end{lem}
 \begin{proof}
  We have 
   $$\gamma \equiv \left( \begin{array}{cc} a & b\\ c& a+(d-a)  \end{array} \right) \equiv \left( \begin{array}{cc} a & 0\\ 0& a  \end{array} \right) \, \; \mathrm{mod} \; u_\gamma,$$
 so in particular $\gamma \; \mathrm{mod} \; p^{v_p(u_\gamma)}$ is scalar. Reciprocally, if $\gamma \; \mathrm{mod} \; p^k$ is scalar for some $k\geq 1$, this means that $p^k$ divides the three coefficients $b,c,d-a$, whose gcd is $u_\gamma$, so $k\leq v_p(u_\gamma)$. Now we remark that if $p\neq 2$ and $k\geq 1$, there are exactly two scalar matrices in $\SL(2,\Z/p^k\Z)$, namely $\pm I_2$.
 \end{proof} 

 We note for later that being scalar mod $p^k$ implies that the trace is $\pm 2$  mod $p^{2k}$:

\begin{lem}\label{pn2separation}
 Let $k\geq 1$ and $p$ prime.  If $\gamma \in \Gamma$ is such that
 $$\gamma=\pm I_2 \; \mathrm{mod} \; p^{k},$$
then 
$$\Tr(\gamma)= \pm 2 \; \mathrm{mod} \; p^{2k}.$$ 

\end{lem}
\begin{proof}
 We may write $\gamma$ as
\begin{equation} \label{expgabcd1}
\gamma=\epsilon\left( I_2 +p^k \begin{pmatrix}a & b \\ c & d \end{pmatrix} \right),
\end{equation}
where $a,b,c,d$ are integers, $\epsilon=\pm 1$,. Since $\gamma$ has determinant $1$, we have
$$ (a+d)p^k +(ad-bc)p^{2k}=0.$$
We factor $p^k$, and reduce modulo $p^k$:
$$a+d =0 \; \mathrm{mod} \; p^{k},$$
Now returning to Equation (\ref{expgabcd1}) and taking the trace, we get
\begin{align*}
\Tr (\gamma) & =\epsilon (2+(a+d)p^k)   \\
& =\pm 2 \; \mathrm{mod} \; p^{2k}.  
\end{align*}
\end{proof}

\subsection{The $p$-adic valuations of $u_\gamma f_\gamma$, $p\neq 2$}

 \begin{lem} \label{ufgamma}
 If $p\neq 2$ is prime, then 
 $$v_p(u_\gamma f_\gamma)=\sup \{k\geq 0 \; : \; \Tr(\gamma)=\pm 2 \; \mathrm{mod} \; p^{2k} \}.$$ 
 \end{lem}
\begin{proof} Let $D$ be the square-free positive integer such that $\K=\Q(\sqrt{D})$, and recall that
$$\Delta_\K=\begin{cases}
 4D &\; \mathrm{if} \; D\equiv 2,3 \, \mathrm{mod} \; 4,\\
 D & \; \mathrm{if} \; D\equiv 1 \, \mathrm{mod} \; 4. 
\end{cases}
$$
 Since $\Tr(\gamma)^2-4=(u_\gamma f_\gamma)^2\Delta_\K$, and $p\neq 2$,
 $$v_p(\Tr(\gamma)^2-4)=\begin{cases}
  1+2v_p(u_\gamma f_\gamma) & \; \mathrm{if} \; D\equiv 0 \, \mathrm{mod} \; p,\\
  2v_p(u_\gamma f_\gamma) & \; \mathrm{if} \; D\not\equiv 0 \, \mathrm{mod} \; p.
 \end{cases}
 $$
But $\Tr(\gamma)^2-4=(\Tr(\gamma)-2)(\Tr(\gamma)+2)$ where 
$\Tr(\gamma)+2$ and $\Tr(\gamma)-2$ have a gcd dividing $4$, so at least one the two terms $\Tr(\gamma)\pm 2$ is not divisible by $p\neq 2$. It follows that:  
$$v_p(\Tr(\gamma)^2-4)=\max(v_p(\Tr(\gamma)-2),v_p(\Tr(\gamma)+2)),$$
so $v_p(u_\gamma f_\gamma)$ is the largest integer $k$ such that one of the two numbers $\Tr(\gamma)\pm 2$ is divisible by $p^{2k}$.

\end{proof} 
 
 We now define for $p\neq 2$ and $\alpha,\beta\geq 0$ the set 
 $$F_{p,\alpha,\beta}=\left\{\gamma \in \SL(2,\Z_p) \, : \, \gamma = \pm I_2 \; \mathrm{mod} \; p^\alpha, \; \Tr(\gamma)= \pm 2  \; \mathrm{mod} \; p^{2\beta} \right\}.$$

 With the convention on congruences mod $p^0$, we have $F_{p,0,0}=\SL(2,\Z_p)$. Clearly, this set is defined mod $p^{\max(\alpha,2\beta)}$ so is clopen in $\SL(2,\Z_p)$. Theses sets are decreasing in both variables: if $\alpha'\geq \alpha, \beta'\geq \beta$, then $F_{p,\alpha',\beta'} \subset F_{p,\alpha,\beta}$. By the Lemmata \ref{ugamma} and \ref{ufgamma}, for any hyperbolic matrix $\gamma \in \Gamma$,
 $$v_p(u_\gamma)\geq \alpha, v_p(u_\gamma f_\gamma)\geq \beta \Leftrightarrow \gamma \in  F_{p,\alpha,\beta}.$$
 Define the clopen set
 $$H_{p,\alpha,\beta}=F_{p,\alpha,\beta}-F_{p,\alpha+1,\beta},$$
 so we have for any hyperbolic matrix $\gamma \in \Gamma$
 $$v_p(u_\gamma)= \alpha, v_p(f_\gamma)\geq \beta-\alpha \Leftrightarrow \gamma \in H_{p,\alpha,\beta}.$$
 Notice that $H_{p,\alpha,\beta'}\subset H_{p,\alpha,\beta}$ if $\beta'\geq \beta$.

 Now for $i\geq 0,j\geq 0$, we define the (still clopen) sets

$$E_{p,\alpha,\beta}=H_{p,\alpha,\alpha+\beta}-H_{p,\alpha,\alpha+\beta+1},$$
 so this time
 $$v_p(u_\gamma)= \alpha, v_p(f_\gamma)= \beta \Leftrightarrow \gamma \in E_{p,\alpha,\beta}.$$
 The sets $(E_{p,\alpha,\beta})_{\alpha,\beta\geq 0}$ are disjoint: if fact we could have defined them directly as the set of matrices $\gamma \in \SL(2,\Z_p)$ such that $\gamma$ is congruent to $\pm I_2$ mod $p^\alpha$ but not $p^{\alpha+1}$, and $\Tr(\gamma)$ is congruent to $\pm 2$ mod $p^{2(\alpha+\beta)}$ but not mod $p^{2(\alpha+\beta)+2}$. But it will be more convenient to decompose them like this for counting purposes. Matrices $\gamma \in \SL(2,\Z_p)$ that do not belong to any $E_{p,\alpha,\beta}$ are either equal to $\pm I_2$, or such that their trace is $\pm 2$. Since the first possibility also imply the second, we have:

  $$\SL(2,\Z_p)-\Tr^{-1}(\pm 2)= \bigsqcup_{\alpha,\beta \geq 0} E_{p,\alpha,\beta},$$
as claimed.  
 
\section{Counting matrices in $\SL(2,\Z/p^k\Z)$}\label{section7}

 In this section, we count in an elementary fashion the number of matrices in $\SL(2,\Z/p^k\Z)$ satisfying some congruence condition, for an odd prime $p$. The objectives are:
 \begin{itemize}
  \item To obtain formulas for the $p$-adic measures $\mu_p(E_{p,\alpha,\beta})$ that are involved in the definition of $q_{n,m}$. The first step will be to compute $\mu_p(F_{p,\alpha,\beta})$ for $p\neq 2$. 
 \item To check that the set of matrices of trace $\pm2$ in $\SL(2,\Z_p)$ is of zero $\mu_p$ measure, which is the last unproven claim of the  first fact of section \ref{3fact}. This is a corollary of the previous step, as one can check using the formulas obtained that 
 $$\sum_{\alpha,\beta\geq 0} \mu_p(E_{p,\alpha,\beta})=1.$$
  \item To check the second claimed fact of section \ref{3fact}, that is the convergence of the infinite product
  $$q_{1,1}=\prod_p \mu_p(E_{p,0,0}).$$
  \item To obtain the formulas for $\sum_{m\geq 1} q_{1,m}$, $\sum_{n\geq 1} q_{n,1}$ appearing in corollaries \ref{cor1} and \ref{cor2}.
 \end{itemize}

 Again, the bothersome case of the prime $p=2$ is treated in the appendix.

\subsection{Preliminaries}

We recall the following:
 \begin{lem} \label{cardinalsl2}
Let $p$ be a prime number, and $k\geq 1$. Then
$$\left| {\bf SL}(2,\Z/p^k\Z)  \right|=p^{3k-2}(p^2-1).$$
\end{lem}
 

\begin{lem} \label{quadric}
 Let $p$ be a prime number and $k\geq 1$. The number of points $(X,Y,Z) \in (\Z/p^k\Z)^3$ on the quadric
$$X^2+YZ=0,$$
is
$$N_{p,k}=
p^{2k}+p^{2k-1}-p^{\floor{\frac{3k-1}2}}.$$
 
\end{lem}
We will need the following:
\begin{proof}
 We count the solutions in an elementary fashion. Given $m\geq 0$, denote by $N_{p,k,m}$ the number of solutions such that $v_p(Y)=m$. Notice first that in the case where $Y=0$, then $X^2=0$ necessarily so $v_p(X)\geq \ceil{k/2}$. Thus if $r=\ceil{k/2}$, $X$ can be written 
 $X=p^rU$ with $U\in \Z/p^{k-\ceil{k/2}}\Z$ arbitrary, and $Z$ can be chosen arbitrarily. Thus
 $$N_{p,k,\infty}=p^{2k-\ceil{k/2}}.$$
 We now assume that $0\leq m \leq k-1$ i.e $Y\neq 0$.
 
 If $(X,Y,Z)$ is a solution, then from $X^2=-YZ$ we get that $2v_p(X)\geq v_p(Y) =m$. Let $r=\ceil{m/2}$ be the least integer larger than $m/2$, thus if $(X,Y,Z)$ is a solution with $v_p(Y)=m$, then there exists $U \in \Z/p^{k-r}\Z$, $V\in (\Z/p^{k-m}\Z)^*$ such that 
 $$X=p^r U, \; Y=p^m V$$
  We can rewrite the equation $X^2+YZ=0$ as
 $$p^{2r}U^2=-p^mVZ \; mod \; p^k,$$
 equivalently
 $$Z =  -p^{2r-m}U^2 V^{-1} \; mod \; p^{k-m},$$
 since $V$ has an inverse mod $p^{k-m}$. Thus given any pair $(U,V) \in (\Z/p^{k-r}\Z)\times(\Z/p^{k-m}\Z)^*$, there are exactly $p^m$ values of $Z\in \Z/p^k\Z$ such that $(X=p^r U,Y=p^m V,Z)$ is a solution to $X^2+YZ=0$. We have obtained that
 $$N_{p,k,m}=p^{k-r}(p^{k-m}-p^{k-m-1}) p^m=
 p^{2k-\ceil{m/2}}-p^{2k-\ceil{m/2}-1}$$
 
 Summing over $m$ gives
 $$N_{p,k}=N_{p,k,\infty}+\sum_{m=0}^{k-1} N_{p,k,m},$$
 Notice that the subsequence for even terms is telescopic, so we get
 $$\sum_{m=2n\leq k-1} N_{p,k,m}=p^{2k} - p^{2k-\floor{(k-1)/2}-1} ,$$
 because if $m=2n$ is even, then $\ceil{m/2}=n$ and the sum runs from $n=0$ to $n=\floor{(k-1)/2}$.
 Similarly for odd terms
 $$\sum_{m=2n-1\leq k-1} N_{p,k,m}=p^{2k-1}  - p^{2k-\floor{k/2}-1},$$
 since if $m=2n-1$ is odd, then $\ceil{m/2}=n$, and the sum runs from $n=1$ to $n=\floor{k/2}$. Note also that in the case where the sum is over an empty set, i.e. $\floor{k/2}<1$, then $k=1$ and the above formula is still valid in this case. 
  Thus,
  $$N_{p,k}=p^{2k}+p^{2k-1}+p^{2k-\ceil{k/2}}-p^{2k-\floor{(k-1)/2}-1}-p ^{2k-\floor{k/2}-1}.$$
  Two of the last three terms cancel out according to the parity of $k$. In both cases, we can write;
  $$N_{p,k}=p^{2k}+p^{2k-1}-p^{\floor{\frac{3k-1}2}}.$$
\end{proof}

\begin{lem} \label{countingf}
Let $\alpha\geq 0,\beta\geq 0$ with $(\alpha,\beta)\neq (0,0)$, and $p$ an odd prime.
The number of matrices $\gamma$ in ${\bf SL}(2,\Z/p^{2(\alpha+\beta)}\Z)$ such that 
$$
\begin{cases}
\gamma = \pm I_2 \; \mathrm{mod} \; p^{\alpha}, \\
\Tr(\gamma)= \pm 2 \; \mathrm{mod} \; p^{2(\alpha+\beta)}
\end{cases}
$$
is $$2p^{3\alpha}(p^{4\beta}+p^{4\beta-1}-p^{3\beta-1}),$$
\end{lem}
\begin{proof}
First of all, since $\Tr(I_2)=+2$, if $\alpha\geq 1$, the sign must be the same in the two defining congruences. If $\alpha=0$ then the first congruence does not matter. So it is enough to count the matrices such that
$$
\begin{cases}
\gamma = + I_2 \; \mathrm{mod} \; p^{\alpha}, \\
\Tr(\gamma)= + 2 \; \mathrm{mod} \; p^{2(\alpha+\beta)}
\end{cases}
$$
and multiply the result by $2$, as $-\gamma$ satisfy opposite congruences, and $-2\neq +2 \; \mathrm{mod} \; p^{2(\alpha+\beta)}$ since $p$ is odd and $\alpha+\beta\geq 1$. \\

We can write any matrix in in ${\bf SL}(2,\Z/p^{2(\alpha+\beta)}\Z)$ satisfying the above congruences as
$$\gamma=\begin{pmatrix} 
1+p^\alpha a & p^\alpha b\\
p^\alpha c & 1 - p^\alpha a
\end{pmatrix},$$
where $(a,b,c) \in \Z/p^{\alpha+2\beta}\Z$. Given such $(a,b,c)$, the matrix above is in $\SL(2,\Z/p^{2(\alpha+\beta)}\Z)$ iff
$$(1-p^{\alpha}a)(1+p^{\alpha}a)-p^{2\alpha}bc=1 \; \mathrm{mod} \; p^{2(\alpha+\beta)},$$
that is
\begin{equation} \label{quadmodp}
a^2+bc=0 \; \mathrm{mod} \; p^{2\beta}.
\end{equation}
By Lemma \ref{quadric}, this quadric has $p^{4\beta}+p^{4\beta-1}-p^{3\beta-1}$ points in $\Z/p^{2\beta}\Z$. Each element of $\Z/p^{2\beta}\Z$ has $p^\alpha$ lifts in $\Z/p^{\alpha+2\beta}\Z$, so there are $p^{3\alpha}(p^{4\beta}+p^{4\beta-1}-p^{3\beta-1})$ triplets $(a,b,c)$ in $\Z/p^{\alpha+2\beta}\Z$ satisfying (\ref{quadmodp}). 
\end{proof}

\subsection{Measure of $E_{p,\alpha,\beta}$, $p\neq 2$}
Recall that we defined for odd primes $p$:
$$F_{p,\alpha,\beta}=\left\{\gamma \in \SL(2,\Z_p) \, : \, \gamma = \pm I_2 \; \mathrm{mod} \; p^\alpha, \; \Tr(\gamma)= \pm 2  \; \mathrm{mod} \; p^{2\beta} \right\}.$$

\begin{lem}\label{formula_mu_f_p}
 Let $p\neq 2$ be prime, and $\alpha\geq 0$, $\beta\geq \alpha$ be integers, not both zero.
$$\mu_p(F_{p,\alpha,\beta})=
\begin{cases}
1 & \; \mathrm{if} \; (\alpha,\beta)=(0,0), \\
\dfrac{2}{p^{3\alpha-2}(p^2-1)}  &\; \mathrm{if} \; (\alpha,\beta)\neq (0,0) , \; \beta\leq \alpha,  \\
\dfrac{2(p^{\beta-\alpha+1}+p^{\beta-\alpha}-1)}{p^{3\beta-1}(p^2-1)} &\; \mathrm{if} \; (\alpha,\beta)\neq (0,0) , \; \beta\geq \alpha.
\end{cases}
$$
\end{lem}
\begin{proof}
 For $(\alpha,\beta)=(0,0)$, by convention, $F_{p,0,0}=\SL(2,\Z_p)$ so 
$$\mu_p(F_{p,0,0})=1.$$
 If $\beta\leq \alpha$, by Lemma \ref{pn2separation}, we have $F_{p,\alpha,\beta}=F_{p,\alpha,\alpha}$; note that the two applicable formulas given in the statement agree when $\beta=\alpha$, so it is sufficient to prove the formula when $\beta\geq \alpha$. 

 In this case,
$$\mu_p(F_{p,\alpha,\beta})=\frac{\left| \{ \gamma \in \SL(2,\Z/p^{2\beta}\Z) \, : \, 
\gamma = \pm I_2 \; \mathrm{mod} \; p^{\alpha},  \Tr(\gamma)= \pm 2 \; \mathrm{mod} \; p^{2\beta}\right| }{\left| \SL(2,\Z/p^{2\beta}\Z) \} \right|},
$$
and by Lemmata \ref{cardinalsl2}  and \ref{countingf}, this ratio is equal to
$$\mu_p(F_{p,\alpha,\beta})=\frac{2p^{3\alpha}(p^{4(\beta-\alpha)}+p^{4(\beta-\alpha)-1}-p^{3(\beta-\alpha)-1})}{p^{3\beta-2}(p^2-1).}.$$
\end{proof}

\begin{prop}\label{formula_for_muE}
 For $p\neq 2$, $\alpha\geq 0$, $\beta\geq 0$,
$$
\mu_p(E_{p,\alpha,\beta})=
\begin{cases}
1-\dfrac{2(p^2+p-1)}{p^{2}(p^2-1)} &\; \mathrm{if} \; (\alpha,\beta)= (0,0), \\
\dfrac{2(p^3+p^2-1)}{p^{3\alpha+2}(p+1)} &\; \mathrm{if} \; \alpha>0, \beta= 0, \\
\dfrac{2(p^2-1)}{p^{3\alpha+2\beta+2}} &\; \mathrm{if} \; \beta> 0.
\end{cases}
$$
\end{prop}
\begin{proof} Recall that $H_{p,\alpha,\beta}$ was defined as 
$F_{p,\alpha,\beta}-F_{p,\alpha+1,\beta}$, and that  
$F_{p,\alpha+1,\beta}\subset F_{p,\alpha,\beta}$, so
$$\mu_p(H_{p,\alpha,\beta})=\mu_p(F_{p,\alpha,\beta})-\mu_p(F_{p,\alpha+1,\beta}).$$
Likewise, 
$E_{p,\alpha,\beta}=H_{p,\alpha,\alpha+\beta}-H_{p,\alpha,\alpha+\beta+1}$ where $H_{p,\alpha,\alpha+\beta+1}\subset H_{p,\alpha,\alpha+\beta}$. 
So we get the  inclusion-exclusion-like formula:
\begin{align*}
\mu_p(E_{p,\alpha,\beta}) = &\mu_p(F_{p,\alpha,\alpha+\beta})+\mu_p(F_{p,\alpha+1,\alpha+\beta+1})\\
&-\mu_p(F_{p,\alpha+1,\alpha+\beta})-\mu_p(F_{p,\alpha,\alpha+\beta+1}).
\end{align*}

 We first consider the case where $\beta>0$. Then $\alpha+\beta \geq \alpha+1$ so we can consistently apply the third formula for the terms $\mu_p(F_{p,\alpha',\beta'})$:
\begin{align*}
\mu_p(E_{p,\alpha,\beta})=& 2\Big(
\frac{ p^3(p^{\beta+1}+p^{\beta}-1)+(p^{\beta+1}+p^{\beta}-1)}{p^{3\alpha+3\beta+2}(p^2-1)}  \\
& -\frac{p^3(p^{\beta}+p^{\beta-1}-1)+(p^{\beta+2}+p^{\beta+1}-1)}{p^{3\alpha+3\beta+2}(p^2-1)} \Big),\\
=&\frac{2(p^{\beta+4}-2p^{\beta+2}+p^\beta)} {p^{3\alpha+3\beta+2}(p^2-1)}\\
=& \frac{2(p^2-1)}{p^{3\alpha+2\beta+2}}.
\end{align*}

We now consider the subcase where $\beta=0$. In this case, the two terms
$\mu_p(F_{p,\alpha+1,\alpha+1})$ and $\mu_p(F_{p,\alpha+1,\alpha})$ cancel each other
so the formula reduces to
$$\mu_p(E_{p,\alpha,0})=\mu_p(F_{p,\alpha,\alpha})-\mu_p(F_{p,\alpha,\alpha+1}).$$
If $\alpha=0$ also, we get
$$\mu_p(E_{p,0,0})=1-\mu_p(F_{p,0,1})=1-\frac{2(p^2+p-1)}{p^2(p^2-1)}.$$
If $\alpha>0$ and $\beta=0$ then
\begin{align*}
\mu_p(E_{p,\alpha,0})&=\frac{2}{p^{3\alpha-2}(p^2-1)}-\frac{2(p^2+p-1)}{p^{3\alpha+2}(p^2-1)},\\
&= \frac{2(p^4-p^2-p+1)}{p^{3\alpha+2}(p^2-1)}, \\
&= \frac{2(p^3+p^2-1)}{p^{3\alpha+2}(p+1)}.
\end{align*}

\end{proof}

\subsection{Corollaries}
 Since $\mu_p(E_{p,0,0})=1+O(p^{-2})$, we have obtained the awaited second fact:
\begin{cor} \label{convergence00}
The product $\prod_{p} \mu_p(E_{p,0,0})$ converges.
\end{cor}

 We leave to the reader to check that $\sum_{\alpha,\beta} \mu_p(E_{p,\alpha,\beta})=1$, which implies that these sets cover $\SL(2,\Z_p)$ up to a set of measure zero, as claimed.
 Let's find the infinite products appearing in the corollaries \ref{cor1} and \ref{cor2}.
\begin{cor}
$$\sum_{n\geq 1} q_{1,m}=\frac56 \prod_{p>3} \left(1- \frac{2}{p(p^2-1)} \right).$$
\end{cor}
\begin{proof}
 We have
$$\sum_{m\geq 1} q_{1,m}=\sum_n \prod_p \mu_p(E_{p,0,v_p(m)})=\prod_p \left( \sum_{\beta\geq 0} \mu_p(E_{p,0, \beta}) \right),$$ 
so we can compute first the sum for a fixed $p$. For $p=2$,  $\sqcup_{\beta \geq 0} E_{2,0,\beta}=A_0$ is the set of matrices not congruent to $I_2$ mod $2$, has measure $5/6$ because out of the 6 matrices in $\SL(2,\Z/2\Z)$, only one is scalar (this can also be obtained but summing the expressions given in proposition \ref{i_hate_two}).
For $p>2$,
$$\sum_{\beta\geq 0} \mu_p(E_{p,0, \beta})=\mu_p(F_{p,0,0})-\mu_p(F_{p,0,1}),$$
quantities for which Lemma \ref{formula_mu_f_p} gives formulas.

\end{proof}

\begin{cor}
$$\sum_{n\geq 1} q_{n,1}= \frac{75}{112} \prod_{p>3} \left(1- \frac{2p}{p^3-1} \right).$$
\end{cor}
\begin{proof} 
 Likewise,
$$\sum_{n\geq 1} q_{n,1}=\prod_p \left( \sum_{\alpha\geq 0} \mu_p(E_{p,\alpha,0}) \right).$$ 
For the prime 2, from Proposition \ref{i_hate_two}, we get by summing over $\alpha$:
$$\mu_2(\sqcup_{\alpha\geq 0} E_{2,\alpha,0})=\frac{75}{112}.$$
For primes $p\geq 3$, we get
$$\sum_{\alpha\geq0} \mu_p(E_{p,\alpha,0})=1-\frac{2(p^2+p-1)}{p^2(p^2-1)}+\frac{2(p^3+p^2-1)}{p^2(p+1)(p^3-1)}=1-\frac{2p}{p^3-1}.$$
\end{proof}

\section{Large prime factors of $u_\gamma f_\gamma$} 

 In this section, we prove the 
 \begin{prop} \label{no_large_prime}
For any $\epsilon>0$, there exist $K(\epsilon)>0$ such that the upper density of the set
 $$M_{K(\epsilon)}=\{ \gamma \in \SL(2,\Z) \, : \; |\Tr(\gamma)|>2, \;  \exists p\geq K(\epsilon) \; \mathrm{prime},  \, p|u_\gamma \; \mathrm{or} \; p|f_\gamma \},$$
 is smaller than $\epsilon$.
 \end{prop}
 \begin{proof} We first fix $K\geq 3$ and compute an upper bound for the upper density of $M_K$; it will be sufficient to check that this bound does tend to zero as $K\to +\infty$. 
 
 If $p>2$ is prime, $\gamma \in \Gamma$ hyperbolic, and $p$ divides either $u_\gamma$ or $f_\gamma$, then since 
 $$\Tr(\gamma)^2-4=u_\gamma^2 f_\gamma^2 \Delta_\K,$$
we see that $p^2$ divides the product $\Tr(\gamma)^2-4=(t-2)(t+2)$, 
where $t:=\Tr(\gamma)$. This means that 
$$t\equiv  \pm 2 \; \mathrm{mod}\; p^2,$$
 because the gcd of $t-2$ and $t+2$ is at most $4$. This congruence of the trace merely depends on the (right) coset of $\gamma$ mod $\Gamma(p^2)$.
 Denote by $Z_p\subset \Gamma$ a set of representatives of cosets mod $\Gamma(p^2)$ whose trace are $\pm 2$ mod $p^2$, so that
 $$M_K \subset \bigcup_{p \geq K} Z_p \Gamma(p^2).$$
 
  Let $T>K$ be large, and $\alpha \in (0,1)$ be a (small) parameter to be chosen later. We cut $M_K\cap \Gamma_T$ in two parts, the part $M_K^{small}$ where $t^2-4$ is divisible by $p^2$ for some relatively small $p\leq T^\alpha$ (but still $p\geq K$), and the part $M_K^{big}$ where $t^2-4$ has a very large square prime factors $p\geq T^\alpha$. \\

  We first analyze $M^{small}_K$. Let $p$ be a small prime $p\leq T^\alpha$, $p\geq K$. By the previous remark,
$$M^{small}_K\cap \Gamma_T \subset \bigcup_{K\leq p \leq T^\alpha} (Z_p \Gamma(p^2) \cap \Gamma_T).$$  
  By the uniform equidistribution Theorem \ref{equid_GN} (\cite{Nevo_Sarnak}), 
 there exist constants $C>0$, $\beta>0$, $T_0>0$ such that for all $p$ and $\gamma_0 \in \Gamma$ and $T\geq T_0$,
$$\left| |\gamma_0 \Gamma(p^2)\cap \Gamma_T| - \frac{|\Gamma_T|}{[\Gamma:\Gamma(p^2)]} \right|\leq C |\Gamma_T|^{1-\beta},$$
so
$$|\gamma_0 \Gamma(p^2)\cap \Gamma_T|\leq \frac{|\Gamma_T|}{[\Gamma:\Gamma(p^2)]}  + C |\Gamma_T|^{1-\beta}.$$

   Thus, adding for each $p \in (K,T^\alpha)$ and each $\gamma_0 \in Z_p$ these upper bounds, we get
   
   $$|M^{small}_K\cap \Gamma_T |\leq \sum_{K\leq p\leq T^\alpha}  |Z_p|\left( \frac{|\Gamma_T|}{[\Gamma:\Gamma(p^2)]} + C |\Gamma_T|^{1-\beta}\right).$$
    
  By Lemma \ref{countingf} applied with $\alpha=0$, $\beta=1$, we get the estimate
$$|Z_p|= 2(p^4+p^3-p^2)\leq 4p^4,$$
 and by Lemma \ref{cardinalsl2}, $[\Gamma:\Gamma(p^2)]=p^4(p^2-1)\geq p^6/2$. Thus
   $$|M^{small}_K\cap \Gamma_T |\leq |\Gamma_T| \left( \sum_{K\leq p\leq T^\alpha}  \frac{8}{p^2} \right) + C \left( \sum_{K\leq p\leq T^\alpha} 4p^4 \right)  |\Gamma_T|^{1-\beta}.$$   
  We can bound $\sum_{K\leq p \leq T^\alpha} 4p^4 \leq 4T^{5\alpha}$. Also, since $|\Gamma_T|\sim 6T^2$, for $T$ large enough, $|\Gamma_T|\geq 5T^2$, so we have
  $$\mathbb{P}_T(M_K^{small}) \leq \frac{8}{K-1} + \frac{C}{5}T^{5\alpha - 2\beta}.$$
  We now fix the choice of $\alpha$, such that $5\alpha<2\beta$. We get the desired conclusion that $M_K^{small}$ has small $\mathbb{P}_T$ measure when $K$ and $T$ are large.\\
  
  Now we consider $M^{big}_K$.
  For fixed $p$, because $|t|\leq 2T$, there are at most $8T/p^2$ possible traces $t$ that satisfies the congruence condition $t\equiv \pm2 \; \mathrm{mod} \; p^2$ (recall that we excluded $t= \pm 2$). By Lemma \ref{bornetrace}, each such trace is realized by at most $c_1 T^{1+\eta}$ matrices in $\Gamma_T$, where $\eta>0$ is chosen such that $\eta<\alpha$. To summarize,
  $$|M_K^{big}|\leq \sum_{T^\alpha \leq p} \frac{8T}{p^2}c_1 T^{1+\eta}\leq c_2 T^{2+\eta } \sum_{p \geq T^\alpha} \frac1{p^2}.$$
 One has $\sum_{p \geq T^\alpha} \frac1{p^2}=O(T^{-\alpha})$, so
  $$|M_K^{big}|\leq c_3 T^{2+\eta-\alpha }.$$
  We thus obtain
  $$\mathbb{P}_T(M_K^{big})=O(T^{\eta-\alpha}).$$
 \end{proof}

\section{Class number}
 In this section, we indicate how to prove Corollary \ref{brauer-siegel}, which states that class numbers are of the order of $T^{1\pm\epsilon}$ with high probability.

 By the Brauer-Siegel Theorem,
 $$\frac{\log \left( h(\Delta_\K) \log(\epsilon_\K)\right) }{\log \sqrt{\Delta_\K}}\to 1,$$
 as the field $\K$ varies. By Theorem \ref{proba_fondamental} and Proposition \ref{distrib_trace}, $\epsilon_\K$ is of the order of $T$ in the sense that for every $\eta>0$, there exists a $c>0$ such that on a set of measure $1-\eta$,
 $$c T<\epsilon_\K<T.$$
By Theorem $2$, $\Delta_K$ is of the order of $T^2$ with high probability, in the same sense. Embedding these two estimates in the Brauer-Siegel Theorem, we obtain that with high probability, $h(\Delta_\gamma)$ is generally of size $T^{1\pm\epsilon}$:
$$\frac{\log \left( h(\Delta_\K) \right) }{\log T}\overset{\mathbb{P}_T} \to 1.$$

 There is a classical formula for $h(\Delta_\gamma)$, see for example \cite[Theorem 12.12]{neukirch2013algebraic}, or Conrad's notes \cite{Conrad}:

$$\frac{h(\mathcal{O}_\gamma)}{h(\mathcal{O}_\K)}=\frac{[\mathcal{O}_\K/f_\gamma \mathcal{O}_\K:
\mathcal{O}_\gamma/f_\gamma \mathcal{O}_\K]}{[\mathcal{O}_\K^\times:\mathcal{O}_\gamma^\times]}. $$

 Here $h(\mathcal{O})$ denotes is the classical (not narrow) class number of the ring $\mathcal{O}$, that may be equal or half the corresponding narrow class number $h(\Delta)$. Anyway,  the formula shows that the ratio $h(\Delta_\gamma)/h(\Delta_\K)$ is bounded if the conductor $f_\gamma$ is. Still by Theorem \ref{maintheorem}, $f_\gamma$ is bounded with large probability, so we also get in this case:
$$\frac{\log \left( h(\Delta_\gamma )\right) }{\log T}\overset{\mathbb{P}_T}\to  1.$$
 
\appendix
\section{The prime $p=2$} \label{appendixA}
In this subsection, we consider the prime $p=2$. The first task is to identify the $2$-adic valuation of $u_\gamma$ and $f_\gamma$ by congruence conditions, in order to identify the sets $E_{2,\alpha,\beta}$. The second task is to compute their measure. These two tasks involve some tedious case-by-case analysis.

\subsection{The $2$-adic valuation of $u_\gamma$} 

 By Lemma \ref{ugamma}, the $2$-adic valuation of $u_\gamma$ is the largest power $2^k$ such that $\gamma \, \mathrm{mod} \, 2^k$ is scalar. However, scalar matrices in this case are not limited to $\pm I_2$:

 \begin{lem}\label{scalarp2}
 If $p=2$, the scalar matrices in $\SL(2,\Z/2^k\Z)$ are
 \begin{itemize}
  \item $I_2$ if $k=1$,
  \item $\pm I_2$ if $k=2$,
  \item $\pm  I_2$ and $\pm( 1 +2^{k-1})I_2$ if $k\geq 3$.  
 \end{itemize}
 \end{lem}

 The proof is left to the reader.
 It will be useful for later to note that being congruent to either $\pm I_2$ or one of these two additional scalar matrices mod $2^k$ can be distinguished by looking at the trace mod $2^{2k}$.
\begin{lem}\label{p2separation}
 Let $k\geq 3$. 
If
 $$\gamma=\pm(1 +2^{k-1})I_2 \; \mathrm{mod} \; 2^{k},$$
then 
$$\Tr(\gamma)= \pm(2+2^{2k-2}) \; \mathrm{mod} \; 2^{2k}.$$ 
\end{lem}
\begin{proof}
 This is pretty similar to the proof of Lemma \ref{pn2separation}. We may write $\gamma$ as
\begin{equation} \label{expgabcd}
\gamma=\epsilon\left( (1 + 2^{k-1})I_2 +2^k \begin{pmatrix}a & b \\ c & d \end{pmatrix} \right),
\end{equation}
where $a,b,c,d$ are integers, and $\epsilon=\pm 1$. Since $\gamma$ has determinant $1$, we have
$$(1 + 2^{k-1} +a 2^k)(1 +  2^{k-1} +d 2^k)-2^{2k}bc=1.$$
So
$$2^{k} +  2^{2k-2} + (a+d)(1+ 2^{k-1})2^k +(ad-bc)2^{2k}=0.$$
We factor $2^k$, then reduce modulo $2^k$ :
$$(1 + 2^{k-2}) + (a+d)(1+ 2^{k-1}) =0 \; \mathrm{mod} \; 2^{k},$$
whose solution for $k\geq 3$ is:
$$a+d=-1+2^{k-2} \; \mathrm{mod} \; 2^{k}.$$
Now returning to Equation (\ref{expgabcd}) and taking the trace, we get
\begin{align*}
\Tr (\gamma) & =\epsilon (2+(a+d+1)2^k)   \\
& =\pm(2+ 2^{2k-2}) \; \mathrm{mod} \; 2^{2k}.  
\end{align*}

\end{proof}

\subsection{The $2$-adic valuation of $u_\gamma f_\gamma$}
 We now turn our attention to the $2$-adic valuation of $u_\gamma f_\gamma$, which can be determined by looking at the trace of $\gamma$ as follows:
 
 \begin{lem} \label{ufgammap2}
 
 Let $p=2$, and $t=\Tr(\gamma)$ where $\gamma \in \Gamma$ is hyperbolic. Then
\begin{itemize}
\item $v_2(u_\gamma f_\gamma)=0$ if and only if $t \, \mathrm{mod} \, 16 \notin 
\{2,6,10,14\}$.
\item $v_2(u_\gamma f_\gamma)=1$ if and only if $t \, \mathrm{mod} \, 16 \in \{6,10\}$.
\item $v_2(u_\gamma f_\gamma)=2$ if and only if $t \, \mathrm{mod} \, 64 \in \{14,30,34,50\}$.
\item For $\alpha \geq 3$, $v_2(u_\gamma f_\gamma)=\alpha$ if and only if $t$ is of the form 
$$\pm (2+2^{2\alpha+1}) \, \mathrm{mod} \, 2^{2\alpha+2},$$ 
$$\pm (2+3.2^{2\alpha}) \, \mathrm{mod} \, 2^{2\alpha+2},$$ 
$$ \mathrm{or} \; \pm (2+2^{2\alpha-2}) \, \mathrm{mod} \, 2^{2\alpha}.$$
\end{itemize}

 \end{lem}
\begin{proof} Let $D$ be the square-free positive integer such that $\K=\Q(\sqrt{D})$, and recall that
$$\Delta_\K=\begin{cases}
 4D \; \mathrm{if} \; D\equiv 2,3 \, \mathrm{mod} (4),\\
 D \; \mathrm{if} \; D\equiv 1 \, \mathrm{mod} (4). 
\end{cases}
$$
 Since $\Tr(\gamma)^2-4=(u_\gamma f_\gamma)^2\Delta_\K$, we have,

\begin{equation} \label{mod16cases}
v_2(t^2-4)=\begin{cases}
 2v_2(u_\gamma f_\gamma) & \; \mathrm{if} \; D\equiv 1 \, \mathrm{mod} \; 4,\\
 2v_2(u_\gamma f_\gamma)+3  &\; \mathrm{if} \; D\equiv 2 \, \mathrm{mod} \; 4,\\
 2v_2(u_\gamma f_\gamma)+2  & \; \mathrm{if} \; D\equiv 3 \, \mathrm{mod} \; 4.
 \end{cases}
\end{equation}

 It will be useful to note the following facts: 
\begin{itemize}
\item The case $D\equiv 2 \, \mathrm{mod} \; 4$ appears if and only if the valuation of $t^2-4$ is odd.
\item If the valuation of $t^2-4$ is even, i.e. $D\neq 2 \, \mathrm{mod} \; 4$, then if we factor out the powers of $2$, i.e. $t^2-4=2^{v_p(t^2-4)}w$, then $w\equiv D \, \mathrm{mod} \; 4$, because $1$ is the only odd square mod $4$.
\end{itemize}

 We first consider $t$ mod $16$. The following array summarizes what can be said using Equation (\ref{mod16cases}), depending on $t$ mod $16$:

$$
\begin{array}{|c||c|c|c|c|}
\hline
t \; \mathrm{mod} \; 16 & \mathrm{odd} &  \pm 6 & 0, \,  \pm 4, \,  8 & \pm 2 \\
\hline
\hline
D \; \mathrm{mod} \; 4 & 1 & 2 & 3 & ?    \\
\hline
v_p(u_\gamma f_\gamma) & 0 & 1 & 0 & \geq 2 \\
\hline
\end{array}
$$
 Let us show how to find this array. If $t \; \mathrm{mod} \; 16$ is odd, then $t^2-4$ is also odd so we must be in the case $D\equiv 1 \, \mathrm{mod} \; 4$ since in all other cases $t^2-4$ has valuation $\geq 2$, by (\ref{mod16cases}). Hence $v_p(u_\gamma f_\gamma)=0$.\\

 If $t \; \mathrm{mod} \; 16$ is equal to $\pm 6$, then $t+2$ and $t-2$ are congruent to 
$8$ and $4$, or $12$ and $8$ depending on the sign; in both cases, one of them has $2$-adic valuation 3 and the other has valuation $2$, so in the end
$$v_2(t^2-4)=5.$$
In the formula (\ref{mod16cases}), $D\equiv 2 \; \mathrm{mod}  \; 4$ is the only case where the valuation of $t^2-4$ is odd, so here we have $2v_2(u_\gamma f_\gamma)+3=5$, so $v_2(u_\gamma f_\gamma)=1$. \\

If $t \, \mathrm{mod} \, 16$ is $0, 4, 8$ or $12$, then $t-2$ and $t+2$ both have $2$-adic valuation equal to $1$, so $v_2(t^2-4)=2$ is even, so $D$ must be $1$ or $3$ mod $4$.
Let's show that in fact $D$ must be $3$ mod $4$. Since $t \, \mod \, 16 \in \{0, 4, 8,12\}$ so $t^2-4 = 12 \;  \mathrm{mod} \; 16$ so there exists an integer $w$ such that
\begin{equation} \label{t2m4m4}
t^2-4=4w \; \mathrm{with}\; w \equiv 3  \;  \mathrm{mod} \; 4.
\end{equation}

  But as noted before, we must have $w\equiv D \;  \mathrm{mod} \; 4$,
so that $D\equiv 3 \;  \mathrm{mod} \; 4$, so by (\ref{mod16cases}), $v_2(u_\gamma f_\gamma)=0$.\\

 We conclude the justification of the array by considering the case where $t \equiv \pm 2 \, \mathrm{mod} \, 16$. In this case, one of the number $t+2,t-2$ has valuation at least 4 and the other exactly 2. So the valuation of $t^2-4$ is at least $6$, and by Equation (\ref{mod16cases}),
$$6\leq v_p(t^2-4)\leq 2v_p(u_\gamma f_\gamma)+3,$$
so $v_p(u_\gamma f_\gamma)\geq 3/2$, but this is an integer so must be $\geq 2$.\\

 The array explains the two first case $v_2(u_\gamma f_\gamma)=0,1$ of the statement.\\

 We now consider in more depth the case where $t \equiv \pm 2 \, \mathrm{mod} \, 16$. In this case, since $\gamma$ is hyperbolic, $t\neq \pm2$ so there exists a smallest integer $k\geq 2$ such that
$$t \neq \pm 2 \; \mathrm{mod} \; 2^{2k+2}.$$
so there exists an integer $\lambda \in \{1,2,3\}$ such that
$$t = \pm (2 + \lambda 2^{2k}) \; \mathrm{mod} \; 2^{2k+2}.$$
 Thus for some integer $\mu$,
$$t = \pm (2 + \lambda 2^{2k}  + \mu 2^{2k+2}) \; \mathrm{mod} \; 2^{2k+4},$$
so
$$t^2= 4 + 4 \lambda 2^{2k} \; \mathrm{mod} \; 2^{2k+4},$$
where most of the terms vanished when we expanded the square because $k\geq 2$. Thus,
$$t^2-4=\lambda 2^{2k+2}\; \mathrm{mod} \; 2^{2k+4}.$$
If $\lambda=2$, then $v_p(t^2-4)=2k+3$ is odd, so we must be in the case where $D\equiv 2 \, \mathrm{mod} \, 4$, by (\ref{mod16cases}), so $v_p(u_\gamma f_\gamma)=k$.\\
Otherwise, $\lambda =1,3$ so $v_p(t^2-4)=2k+2$ is even so $D\equiv 1,3 \, \mathrm{mod} \, 4$. Again, $(t^2-4)/2^{v_2(t^2-4)}$ has the same residue mod $4$ as $D$, so
$$ D \equiv \lambda \; \mathrm{mod} \; 4.$$
Still by  (\ref{mod16cases}), this implies that if $\lambda=1$, then $v_p(u_\gamma f_\gamma)=k+1$, and if $\lambda=3$ then $v_p(u_\gamma f_\gamma)=k$.\\

 Now we can explain the last two lines of the statement. Assume $v_2(u_\gamma f_\gamma)=\alpha \geq 2$. From the above array, $t$ must be $\pm2$ mod $16$. 

Define $k,\lambda$ as above, then $\alpha=k$ if $\lambda=2,3$ and $\alpha=k+1$ if $\lambda=1$. Since $k\geq 2$, in the case $\alpha=2$, we cannot have $\lambda=1$, and so $k=2$ and thus
$$\alpha=2 \Rightarrow t=\pm (2 + 16 \lambda ) \; \mathrm{mod} \; 64, \; \lambda \in \{2,3\},$$
and the reverse implication follows from the previous analysis.
If $\alpha\geq 3$, the previous analysis also concludes.

\end{proof}

\subsection{Definition of $E_{2,\alpha,\beta}$}\label{defEp2}

 We are now in position to describe the sets $E_{2,\alpha,\beta}$. We first define the clopen sets
$$A_0=\Big\{\gamma \in \SL(2,\Z_2) \, : \, \gamma \neq I_2 \; \mathrm{mod} \; 2\Big\},
$$
$$A_1=\Big\{\gamma \in \SL(2,\Z_2) \, : \, \gamma = I_2 \; \mathrm{mod} \; 2,
 \gamma \neq \pm I_2 \; \mathrm{mod} \; 4
\Big\},$$
$$A_2=\Big\{\gamma \in \SL(2,\Z_2) \, : \, \gamma = \pm I_2 \; \mathrm{mod} \; 4,
 \gamma \neq \pm I_2, \pm 3 I_2  \; \mathrm{mod} \; 8
\Big\},$$
and for $\alpha\geq 3$,
$$A_\alpha=\Big\{\gamma \in \SL(2,\Z_2) \, : \, \gamma = \pm I_2 \; \mathrm{mod} \; 2^\alpha,
 \gamma \neq \pm I_2, \pm (1+2^{\alpha} I_2)  \; \mathrm{mod} \; 2^{\alpha+1}
\Big\},$$
$$B_\alpha=\Big\{\gamma \in \SL(2,\Z_2) \, : \, \gamma =\pm (1+2^{\alpha-1}) I_2 \; \mathrm{mod} \; 2^{\alpha}\Big\},$$
with the convention that $B_0=B_1=B_2=\emptyset$. Thus by Lemmata \ref{ugamma} and \ref{scalarp2}, for any $\gamma \in \Gamma$ hyperbolic,
$$v_2(u_\gamma)=\alpha \Leftrightarrow \gamma \in A_\alpha \sqcup B_\alpha.$$
The sets $A_k,B_k$ form a partition of $\SL(2,\Z_p)-\{\pm I_2\}$.

Similarly, we put:
$$C_0=\Big\{\gamma \in \SL(2,\Z_2) \, : \, \Tr(\gamma) \; \mathrm{mod} \;  16 \notin \{2,6,10,14\} \Big\},$$
$$C_1=\Big\{\gamma \in \SL(2,\Z_2) \, : \, \Tr(\gamma) \; \mathrm{mod} \;  16 \in \{6,10\} \Big\},$$
$$C_2=\Big\{\gamma \in \SL(2,\Z_2) \, : \, \Tr(\gamma) \; \mathrm{mod} \;  64 \in \{14,30,34,50\} \Big\},$$
and for $k\geq 3$,
$$
C_k=\Big\{\gamma \in \SL(2,\Z_2) \, : \, \Tr(\gamma) =
\pm (2+\lambda 2^{2k})   \; \mathrm{mod} \;  2^{2k+2}, \,\lambda\in\{2,3\} \Big\}.
$$
$$D_k=\Big\{\gamma \in \SL(2,\Z_2) \, : \, \Tr(\gamma) = 
\pm(2+2^{2k-2}) \; \mathrm{mod} \;  2^{2k} \Big\}$$

By Lemma \ref{ufgammap2},  for $\gamma$ hyperbolic,
$$v_2(u_\gamma)+v_2(f_\gamma)=k \Leftrightarrow \gamma \in C_{k}\sqcup D_k.$$
The sets $C_k$, $D_k$ form a partition of $\SL(2,\Z_p)-\Tr^{-1}(\pm2)$.
Finally, we put
$$E_{2,\alpha,\beta}=(A_\alpha \sqcup B_\alpha) \cap (C_{\alpha+\beta}\sqcup D_{\alpha+\beta}).$$
So for $\gamma$ hyperbolic,
$$v_2(u_\gamma)=\alpha,  v_2(f_\gamma)=\beta \Leftrightarrow \gamma \in E_{2,\alpha,\beta}.$$
By Lemma \ref{p2separation}, we have for $\alpha\geq 3$, $B_\alpha\subset D_\alpha$. This imply that $B_\alpha \cap C_{\alpha+\beta}=\emptyset$ for all $\beta$. So $E_{2,\alpha,\beta}$ admits a slightly simpler expression:
$$E_{2,\alpha,\beta}=
\begin{cases} 
(A_\alpha \cap C_{\alpha}) \sqcup B_\alpha \; & \mathrm{if} \; \beta=0,\\
A_\alpha \cap (C_{\alpha+\beta} \sqcup D_{\alpha+\beta})\; & \mathrm{if} \; \beta>0.
\end{cases}
$$

\subsection{The measure of $E_{2,\alpha,\beta}$}

The goal of this section is to explain how to get the following array.
\begin{prop}\label{i_hate_two}
The value of $\mu_2(E_{2,\alpha,\beta})$ is given by the following array.
$$
\begin{array}{|c||c|c|c|c|}
\hline
\mu_2(E_{2,\alpha,\beta}) & \alpha=0 &  \alpha=1 & \alpha=2  & \alpha\geq 3 \\
\hline
\hline
\beta=0 & 7/12 &  1/16 &  1/64 & 11/(3*2^{3\alpha}) \\
\hline
\beta=1 & 1/8 &  1/32 &  3/2^{3\alpha+2\beta} & 3/2^{3\alpha+2\beta} \\
\hline
\beta=2 & 1/16 &  3/2^{3\alpha+2\beta}  &  3/2^{3\alpha+2\beta} & 3/2^{3\alpha+2\beta} \\
\hline
\beta\geq 3 & 3/2^{3\alpha+2\beta}&  3/2^{3\alpha+2\beta}  &  3/2^{3\alpha+2\beta} & 3/2^{3\alpha+2\beta} \\
\hline
\end{array}$$

\end{prop}

 For the six cases $\alpha+\beta\leq 2$, these measures can be determined by a computer enumeration of the 196608 matrices in $\SL(2,\Z/64\Z)$, as by Lemma \ref{ufgammap2}, these clopen sets $E_{2,\alpha,\beta}$ are defined modulo $64$. This being done,
 it remains to show the formulas in the two cases $\alpha\geq 3, \beta=0$ and $\alpha+\beta\geq 3, \beta>0$. We first compute the number of matrices which correspond to subsets of $(A_\alpha\sqcup B_{\alpha+1})\cap C_{k}$ or $(A_\alpha\sqcup B_{\alpha+1})\cap D_{k+1}$:

\begin{lem} \label{boring_enumeration}
Let $p=2$, $k\geq \alpha \geq 0$, with $k\geq 2$. Let $\lambda \in \{1,2,3\}$. The number of matrices in $\SL(2,\Z/p^{2k+2}\Z)$ that satisfy the 3 conditions

$$\begin{cases}
 \gamma = \pm I_2 \; \mathrm{mod} \; p^\alpha, \\
 \gamma \neq \pm I_2 \; \mathrm{mod} \; p^{\alpha+1}, \\
\Tr(\gamma) = \pm ( 2+\lambda p^{2k} ) \; \mathrm{mod} \; p^{2k+2},
\end{cases}$$

is 
$$\begin{cases}
5p^{3k+3} &\; \mathrm{if} \; k=\alpha, \,\lambda=1\\
3p^{3k+3} &\; \mathrm{if} \; k=\alpha, \, \lambda=2 \, \mathrm{or}\; 3,\\
3p^{4k-\alpha+3}    &\; \mathrm{otherwise} 
\end{cases}$$
\end{lem}
\begin{proof} For a matrix $\gamma$ satisfying the above condition, its trace is $\pm 2$ 
mod $p^{2k}$. Since $k\geq 2$, $-2\neq +2 \; \mathrm{mod} \; p^{2k}$ so the sign $\epsilon\in\{ \pm 1\}$ such that 
$$\Tr(\gamma)\equiv \epsilon 2 \; \mathrm{mod} \; p^{2k},$$
is defined without ambiguity. 
If $\alpha\geq 2$, the sign $\pm$ in the first condition:
$$ \gamma = \pm I_2 \; \mathrm{mod} \; p^\alpha,$$
is well-defined, and considering the trace mod $p^\alpha$, must agree with the above $\epsilon$. Otherwise, this sign is irrelevant and can be chosen to agree with the above $\epsilon$. So in all cases we are left to count, given a sign $\epsilon \in \{\pm1\}$, the number of matrices satisfying 
$$\begin{cases}
 \gamma = \epsilon I_2 \; \mathrm{mod} \; p^\alpha, \\
 \gamma \neq \epsilon I_2 \; \mathrm{mod} \; p^{\alpha+1}, \\
\Tr(\gamma) = \epsilon ( 2+\lambda p^{2k} ) \; \mathrm{mod} \; p^{2k+2},\\
\det(\gamma)=1
\end{cases}$$

The first condition may be written as the existence of $a,b,c,d$ in $\Z/p^{2k+2-\alpha}\Z$ such that
\begin{equation} \label{expgabcd2}
\gamma=\epsilon\left( I_2 +p^\alpha \begin{pmatrix}a & b \\ c & d \end{pmatrix} \right),
\end{equation}
and the second condition states that not all of them are divisible by $2$. By the third condition on the trace, 
$$d= \lambda p^{2k-\alpha} -a.$$
Since $k\geq \alpha$ and $k\geq 2$, by the above equation, $d$ and $a$ have the same parity; in particular the condition that not all $a,b,c,d$ are even is satisfied iff not all of $a,b,c$ are even.

 We wish to rewrite the condition that $\gamma$ has determinant $1$, as an equation on $a,b,c$, assuming that $d$ is given as above. 
$$ (1+a p^\alpha)(1+\lambda p^{2k}-ap^\alpha)-bc p^{2\alpha}=1 \; \mathrm{mod} \; p^{2k+2},$$
which is equivalent to
\begin{equation} \label{determinant}
 (1+a p^\alpha)\lambda p^{2(k-\alpha)}-a^2-bc=0 \; \mathrm{mod} \; p^{2(k-\alpha)+2}.
\end{equation}
\underline{First case:} $k=\alpha$, so here $\alpha\geq 2$. Then Equation (\ref{determinant}) is equivalent to:
$$\lambda-a^2-bc=0 \; \mathrm{mod} \;  4.$$
We can count (by computer, or by hand) the number of solution $(x,y,z) \in (\Z/4\Z)^3$ to this equation
$x^2+yz=\lambda \; \mathrm{mod} \;  4$ depending on $\lambda$, adding the condition that $x,y,z$ cannot all be even: there is  $20$ solutions if $\lambda=1$, and $12$ if $\lambda=2$ or $3$. 
However, $a,b,c$ are in fact elements of $\Z/p^{k+2}\Z$ whose reduction mod $4$ are a solution $(x,y,z)$, so there are $20p^{3k}$ or $12p^{3k}$ solutions $(a,b,c)$ depending on $\lambda=1$ or $2,3$, for a given sign $\epsilon$. All in all, this gives us $5p^{3k+3}$ or $3p^{3k+3}$ solutions.
\\

\noindent\underline{Second case:} $k>\alpha$. We put $\beta=k-\alpha>0$.
Then, by (\ref{determinant}), $(a,b,c)$ has to be congruent mod $p^{2\beta}$ to a solution, say $a_0,b_0,c_0$, of 
\begin{equation}\label{quadrictwo}
X^2+YZ=0 \; \mathrm{mod} \;  p^{2\beta},
\end{equation}
with the additional requirement that $a_0,b_0,c_0$  cannot be all even. With the help of Lemma \ref{quadric}, one can show that there are 
$p^3(p^{4\beta-4}+p^{4\beta-5}-p^{3\beta-4})$ solutions of (\ref{quadrictwo}) for which $a_0,b_0,c_0$ are all even, among the 
$p^{4\beta}+p^{4\beta-1}-p^{3\beta-1}$ general solutions of (\ref{quadrictwo}). This means that 
$(a_0,b_0,c_0) \in \Z/p^{2\beta}\Z$ are freely chosen in a set of cardinality
$$p^{4\beta-2}(p^2-1).$$
We may now define again auxiliary variables $(u,v,w)\in (\Z/4\Z)^3$ such that
$$a=a_0'+up^{2\beta}, \, b=b_0'+vp^{2\beta}, \, c=c_0'+wp^{2\beta},$$
with $a_0',b_0',c_0'$ being fixed representatives of $a_0,b_0,c_0$ in $\Z/p^{2\beta+2}\Z$.
By construction, there exists some $\nu \in \Z/4\Z$  such that $a_0'^2+b_0'c_0'=\nu p^{2\beta}$.

Equation (\ref{determinant}) can be rewritten as an equation on $u,v,w$:
$$\lambda (1+a_0' p^\alpha)p^{2\beta}=\nu p^{2\beta}+
p^{2\beta}(2a_0'u+b_0'w+c_0'v)
\; \mathrm{mod} \; p^{2\beta+2},$$
so is equivalent to 
\begin{equation}\label{determinantetoile}
\lambda (1+a_0' p^\alpha)-\nu=2a_0'u+b_0'w+c_0'v \; \mathrm{mod} \; p^{2},
\end{equation}
which is linear in $(u,v,w)$. However, inspecting again the equation $X^2+YZ=0$, we see that if $a_0$ is odd then $b_0$ and $c_0$ must also be odd. Since $a_0,b_0,c_0$ are not all even, at least one of the two numbers $b_0$ or $c_0$ must be odd, so the linear map $(\Z/4\Z)^3\to \Z/4\Z$, $(u,v,w)\mapsto 2a_0'u+b_0'w+c_0'v$ must be a surjective group morphism, in this case a 16-to-1 map. This means that the number of solutions $(u,v,w)$ of (\ref{determinantetoile}) is 16.
Thus, the number of solutions $(a,b,c)\in \Z/p^{2\beta+2}\Z$ of (\ref{determinant}) is 
$16*p^{4\beta-2}(p^2-1)=p^{4\beta+2}(p^2-1)$. However, $(a,b,c)$ are in fact chosen in $\Z/p^{k+\beta+2}\Z$, so we must multiply this number by $p^{3\alpha}$. Remembering that we also had to choose a sign $\epsilon$, we get that the number we seek is

$$3p^{4k-\alpha+3}.$$
\end{proof}

 The next two lemma establish the remaining cases of Proposition \ref{i_hate_two}.

\begin{lem} For $\alpha\geq 3$,
$$\mu_p(E_{2,\alpha,0})=\frac{11}{3}2^{-3\alpha}$$ 
\end{lem}
\begin{proof} Recall that for $\alpha\geq 3$, 
$$E_{2,\alpha,0}=(A_\alpha \cap C_\alpha)\sqcup B_\alpha.$$
The measure of $B_\alpha$ is straightforward to compute: among the $p^{3\alpha-2}(p^2-1)$ matrices of $\SL(2,\Z/p^\alpha\Z)$, we must count the two matrices $\pm(1+2^{\alpha-1})I_2$. So
$$\mu_2(B_\alpha)=\frac{1}{(p^2-1)p^{3\alpha-3}}.$$

 We are left to determine the measure of $A_\alpha\cap C_\alpha$; this is almost the previous lemma for $k=\alpha$ and $\lambda =2,3$, except we did not a priori excludes matrices that may be congruent to $\pm(1+2^{\alpha})I_2  \; \mathrm{mod} \; p^{\alpha+1}$. However, such matrices are in $B_{\alpha+1}\subset D_{\alpha+1}$ so never belong to $C_k$ for any $k$. Thus, summing the number of solution for $\lambda=2$ and $\lambda=3$, we get:
$$\mu_2(A_\alpha\cap C_\alpha)=\frac{3p^{3\alpha+3}+3p^{3\alpha+3}}{p^{6\alpha+4}(p^2-1)}
=\frac1{p^{3\alpha}}.$$
So we get
$$\mu_2(E_{2,\alpha,0})=\frac{11}{3p^{3\alpha}}.$$
\end{proof}
 
\begin{lem} For $\beta\geq 1$ and $\alpha+\beta \geq 3$,
$$\mu_2(E_{2,\alpha,\beta})=\frac{3}{p^{3\alpha+2\beta}}.$$
\end{lem}
\begin{proof}
 Lemma \ref{boring_enumeration} with $\lambda=2,3$ and $k=\alpha+\beta$ tells us how to compute the measure of
$(A_\alpha\sqcup B_{\alpha+1})\cap C_{\alpha+\beta}=A_\alpha \cap C_{\alpha+\beta}$.
Adding the cases $\lambda=2,3$, we get that for $\beta>0$,
$$\mu_2(A_\alpha \cap C_{\alpha+\beta})
=\frac{6p^{3\alpha+4\beta+3}}{p^{3(2\alpha+2\beta+2)-2}(p^2-1)}=\frac{1}{p^{3\alpha+2\beta}}.$$

 We now make a distinction between the cases $\beta>1$ and $\beta=1$. First assume that $\beta>1$.
  Lemma \ref{boring_enumeration} with $\lambda=1$, $k=\alpha+\beta-1$  gives us
$$\mu_2((A_\alpha\sqcup B_{\alpha+1})\cap D_{\alpha+\beta})
=\frac{3p^{3\alpha+4\beta-1}}{p^{3(2\alpha+2\beta)-2}(p^2-1)}=\frac{2}{p^{3\alpha+2\beta}},$$
where we used the last formula of Lemma \ref{boring_enumeration} since $k>\alpha$.
Also, since $\beta>1$, $B_{\alpha+1}\cap D_{\alpha+\beta}=\emptyset$ since $B_{\alpha+1}\subset D_{\alpha+1}$, so
$$\mu_2(E_{2,\alpha,\beta})= \mu_2(A_\alpha \cap C_{\alpha+\beta})+\mu_2(A_\alpha \cap D_{\alpha+\beta})=\frac{3}{p^{3\alpha+2\beta}},$$
as claimed.\\

We now treat the case where $\beta=1$. Then Lemma \ref{boring_enumeration} with $\lambda=1$, $k=\alpha$  gives us
$$\mu_2((A_\alpha\sqcup B_{\alpha+1})\cap D_{\alpha+1})
=\frac{5p^{3\alpha+3}}{p^{3(2\alpha+2)-2}(p^2-1)}=\frac{5}{p^{3\alpha+1}(p^2-1)},$$
Since $\beta=1$, we have
$$(A_\alpha\sqcup B_{\alpha+1})\cap D_{\alpha+1}=A_\alpha \cap D_{\alpha+1} \sqcup B_{\alpha+1},$$
so 
$$\mu_2(A_\alpha\cap D_{\alpha+1})=\frac{5}{3p^{3\alpha+1}}-\frac{1}{(p^2-1)p^{3\alpha}}=\frac{1}{p^{3\alpha+1}},$$
so using $E_{2,\alpha,1}=(A_\alpha\cap D_{\alpha+1})\sqcup (A_\alpha\cap C_{\alpha+1})$, the result is
$$\mu_2(E_{2,\alpha,1})=\frac{1}{p^{3\alpha+1}}+\frac{1}{p^{3\alpha+2}}=\frac{3}{p^{3\alpha+2\beta}}.$$
\end{proof}
 
 \section{Enumeration of geodesics by length}
\label{appendixB}
 Here we wish to discuss in more depth the relationship between picking randomly a periodic geodesic on the modular surface using the uniform measure on those of length $\leq L$, and picking a matrix at random using $\mathbb{P}_T$, which is likely hyperbolic and primitive, and then considering its conjugacy class in $\PSL(2,\Z)$, which then correspond to a random periodic geodesic. The point is that these two probabilities are, in some sense, absolutely continuous to each other.\\

 Let $E_\Gamma$ be a subset of $\Gamma$ that is a union of conjugacy classes of primitive hyperbolic elements, symmetric in the sense that if $\gamma \in E_\Gamma$, then $-\gamma \in E_\Gamma$. Then $E_\Gamma$ correspond to a subset $E_\mathcal{G}$ of the set of oriented, primitive periodic geodesic on the modular surface. Denote by $\mathcal{G}_L$ the subset of oriented, primitive periodic geodesic on the modular surface of length $\leq L$.

 \begin{thm} \label{negligiblelength}
 Let $E_\Gamma \subset \Gamma$, $E_\mathcal{G}$ be as above. Then
 $E_\Gamma$ has zero natural density if and only if $E_\mathcal{G}$ has zero natural density. More precisely,
 $$\left( \limsup_{T\to +\infty} \frac{|E_\Gamma \cap \Gamma_T|}{|\Gamma_T|}>0 \right) \Leftrightarrow  \left( \limsup_{L\to +\infty} \frac{|E_\mathcal{G} \cap \mathcal{G}_L|}{|\mathcal{G}_L|}>0 \right).$$
 \end{thm} 
  
 Thus properties of conjugacy classes which are true for almost all (resp. almost no) matrices of $\SL(2,\Z)$, ordered by norm, are also true for almost all (resp.  almost no) of the corresponding periodic geodesics, ordered by length. The rest of the section is devoted to  prove Theorem \ref{negligiblelength}.

 \subsection{Enumeration by norm and displacement}
 
 It is first worth a reminder that the enumeration by norm is nothing else than a enumeration by displacement. Indeed, consider the point $i\in \mathbb{H}^2$ in the upper half-plane model of the hyperbolic plane, then the Frobenius norm of $\gamma$ and the hyperbolic distance between $i$ and $\gamma i$ are related by (see for example \cite[3.1]{MR3434881})
 $$\| \gamma \|^2=2 \cosh d_{\mathbb{H}^2}(\gamma i,i),$$
 where $\left\| \left( \begin{array}{cc} a & b\\ c& d  \end{array}\right) \right\|=\sqrt{a^2+b^2+c^2+d^2}$ is the Frobenius norm. So if $T$ and $R$ are related by $T^2=2\cosh R$, then 
$$\Gamma_T=\left\{\gamma \in \Gamma \; : \; d_{\mathbb{H}^2}(\gamma i,i)\leq R\right\}.$$
 The following lemma is an exercise in hyperbolic geometry that follows for example from thinness of triangles:
 \begin{lem} Let $\gamma \in \Gamma$ be hyperbolic, and $\mathcal{A}_\gamma\subset  \mathbb{H}^2$ be the axis of the hyperbolic transformation $\gamma$ acting on $\mathbb{H}^2$. There is a constant $c>0$ such that 
$$2d_{\mathbb{H}^2}(i,\mathcal{A}_\gamma)+\ell_\gamma \leq d_{\mathbb{H}^2}(\gamma i,i) +c.$$
 \end{lem}

  We now prove the implication $\Rightarrow$. Let $E_\Gamma$ be a set of positive upper density, say $\epsilon$, for the enumeration by norm. The main point is to bound how many times we can pick different elements of $\Gamma_T$ corresponding to the same periodic geodesic, linearly in the length.

By the formula (\ref{distrib_length}), we can find a $\beta=\beta_\epsilon>0$ such that 
$$\limsup_{T\to +\infty} \mathbb{P}_T\left(\gamma \in E_\Gamma, \, 2\log(T) \geq \ell_\gamma \geq 2\log(T)-\beta \right)>\epsilon/2.$$

 Denote by $E'_T=\left\{\gamma \in E_\Gamma \, : \, 2\log T \geq \ell_\gamma \geq 2\log(T)-\beta \right\}$. For any $\gamma \in E'_T$, by the lemma,
$$d_{\mathbb{H}^2}(i,\mathcal{A}_\gamma)\leq \frac{1}2 \left(R(T)-2\log(T)+\beta+c \right),$$
where $R(T)=\cosh^{-1}(T^2/2)=2\log(T)+o(1)$. Thus there is a uniform bound $r=r_\epsilon>0$ on the distance between $i$ and the axis $\mathcal{A}_\gamma$ of $\gamma$ for all elements of $E'_T$. Pick now some $\gamma_0 \in E'_T$, and let $m_{\gamma_0}$ be the number of $\gamma \in E'_T$ which are conjugate to $\gamma$ or $-\gamma$.
 Consider the following picture: draw a disk of radius $r+1$ centered on $i$, and its intersection with the axis $\mathcal{A}_\gamma$ for all $\gamma \in E'_T$ conjugate to $\pm \gamma_0$. What we see is a family of $m_{\gamma_0}$ geodesic segments of length at least $2$. Let $N=N_\epsilon$ be the number of fundamental domains for $\Gamma$ intersecting the disk of radius $r+1$. Then if we project the picture to $\Gamma \backslash \mathbb{H}^2$, we see that
$$2 m_{\gamma_0} \leq N \ell_{\gamma_0}.$$
This is the required linear bound on $m_{\gamma_0}$ in terms of $\ell_{\gamma_0}$. Thus, if for some large $T>0$, $\mathbb{P}_T(E'_T)>\epsilon/4$, meaning there are at least 
$\epsilon T^2$ matrices in $E'_T$, there are representing in at least $\frac2N \epsilon T^2/(2\log(T))$ distinct conjugacy classes. Now put $L=2\log(T)$, using the well-known asymptotic
$$|\mathcal{G_L}|\sim \frac{e^L}L=\frac{T^2}{2\log(T)},$$
we conclude that $E_\mathcal{G}$ must have upper density at least $\frac2N \epsilon$.

 \subsection{Enumeration by length}
We study the converse implication. Assume that $E_\mathcal{G}$ has positive density.
This time, the point is that for most geodesics, one can construct a number of lift close to $i$ which is at least linear in the length. Choose a compact subset $C\subset \Gamma\backslash G$ of Haar measure bigger than $1/2$, and let $\tilde{C}$ be a compact set in $G$ surjecting onto $C$.
By equidistribution of periodic geodesics on the modular surface (see for example 
 \cite{roblin:hal-00104838}), the set of geodesics that spend a proportion of time $\leq 1/3$ in $C$ is of zero density. Thus the set of geodesics
in $E_\mathcal{G}$ that spend a proportion of time $> 1/3$ in $C$ is still of positive upper density. We will thus assume that all geodesics in $E_\mathcal{G}$ satisfy this condition.\\

 Let $M$ be the diameter of $\tilde{C}$. For such a geodesic $\lambda$, one can pick vectors $v_1,...,v_k$ on $\lambda$, with $v_i \in C$,  such that $v_i,v_j$ are image of each other by the geodesic flow of time at least $M+1$ for $i\neq j$. It is possible to do this with $k\geq \ell(\lambda)/3(M+1)$. Pick now lifts $\tilde{v}_i \in \tilde{C}$ of each $v_i$; the geodesics in $\mathbb{H}^2$ defined by the vectors $\tilde{v_i}$ are all lifts of $\lambda$, at bounded distance from the origin $i$. They are distinct since otherwise, $\tilde{v}_i$ and $\tilde{v}_j$ would be at distance $\leq M$ on the same geodesic, so one could flow from one to the other in time $<M$.\\
Thus, for each geodesic $\lambda$ of length $\ell$ that spend a proportion of time $\geq 1/3$ in $C$, one could construct $c.\ell$ different lifts at bounded distance from the point $i$, for some $c>0$. Now pick some large $L$ such that 
$$\frac{|E_\mathcal{G}\cap \mathcal{G}_L |}{|\mathcal{G}_L |}>\epsilon.$$
By exponential growth, there exists $\eta>0$ such that the set
$$E'_L=\{\lambda \in E_\mathcal{G}\cap \mathcal{G}_L \, : \, \ell_\lambda>\eta L\},$$
satisfies $|E'_L|>\epsilon |\mathcal{G}_L |/2\geq \epsilon e^L/3L$, and each of these geodesics has at least $c\eta L$ different lifts at bounded distance from $i$. For each of those $c\eta \epsilon e^L$ lift, there exists an hyperbolic transformation $\gamma \in \Gamma$ whose axis is precisely the lift. We have to show that they all lie in $\Gamma_T$ for some $T$ not far from $e ^{L/2}$. By the triangle inequality,
$$ d_{\mathbb{H}^2}(\gamma i,i)  \leq \ell_\gamma + 2d_{\mathbb{H}^2}(i,\mathcal{A}_\gamma)\leq L+O(1),$$ 
so we have $\gamma \in \Gamma_T$ with $T<Ce^{L/2}$, for some well-chosen $C>0$. 

\section*{Acknowledgement}
 This research was partially funded by the European Research Council (ERC GOAT 101053021).
\bibliography{biblio}{}
\bibliographystyle{plain}

\end{document}